\documentclass[11pt,a4paper]{article}
\usepackage{amsmath, amsfonts, amssymb,amsthm}
\usepackage{a4wide}
\usepackage{parskip}
\usepackage{enumitem}
\usepackage{xcolor}
\usepackage{mathrsfs}
\usepackage{cite}
\usepackage{cleveref}
\providecommand{\texorpdfstring}[2]{#1}
\usepackage{latexsym}

\usepackage[left=1in, right=1in,top=1in,bottom=1in]{geometry}
\allowdisplaybreaks

\newcommand{\be} {\begin{equation}}
	\newcommand{\ee} {\end{equation}}
\newcommand{\bea} {\begin{eqnarray}}
	\newcommand{\eea} {\end{eqnarray}}
\newcommand{\Bea} {\begin{eqnarray*}}
	\newcommand{\Eea} {\end{eqnarray*}}
\newcommand{\noi} {\noindent}

\def\dx{\,{\rm d}x}

\def\R{{\mathbb R}}

\def\R{{\mathbb R}}

\def\t{{2^{\sharp}}}
\newcommand{\sr}{\mathcal{S}(a,b)}

\def\r{{(\rho_1,\rho_2)}}
\def\u{{(u,v)}}
\def\v{{(u_n,v_n)}}
\def\b{\mathcal{B}_{p,q}}
\def\a{\left(\||x|^{-a}\nabla u\|_2^2+\||x|^{-a}\nabla v\|_2^2\right)}
\def\an{\left(\||x|^{-a}\nabla u_n\|_2^2+\||x|^{-a}\nabla v_n\|_2^2\right)}
\def\p{(p+q)_c}
\def\pq{(p+q)\delta_{p+q}}
\def\y{\mathcal{X}}
\def\rt{(\tilde{\rho}_1,\tilde{\rho}_2)}
\newcommand{\ra} {\rightarrow}
\newcommand{\na} {\nabla}

\newcommand{\var} {\varepsilon}

\numberwithin{equation}{section}

\newtheorem{theorem}{Theorem}[section]
\newtheorem{rem}{Remark}[section]
\newtheorem{lemma}{Lemma}[section]
\newtheorem{prop}{Proposition}[section]

\numberwithin{equation}{section}

\begin{document}
	\setlength{\abovedisplayskip}{3pt}
	\setlength{\belowdisplayskip}{3pt}
	\date{}
	\title{
    Existence and Asymptotic Behavior of Normalized Ground States for a Weighted Elliptic System with Caffarelli--Kohn--Nirenberg Critical Exponent}
    \author{ {\bf Asmita Rai$\,$\footnote{e-mail: {\tt asmita.rai65@gmail.com}}, Divya Goel$\,$\footnote{e-mail: {\tt divya.mat@iitbhu.ac.in}}} \\ Department of Mathematical Sciences, Indian Institute of Technology (BHU),\\ Varanasi 221005, India.}
	\maketitle
\begin{abstract}
In this paper, we study the coupled singular weighted elliptic system
		$$\left\{
		\begin{aligned}
			-\operatorname{div}(|x|^{-2a}\nabla u)+\lambda_1 \frac{u}{|x|^{2a}}&=\beta p\frac{|v|^q|u|^{p-2}u}{|x|^{b(p+q)}}+\frac{|u|^{2^{\sharp}-2}u}{|x|^{b2^{\sharp}}},\quad\text{in}~\R^N,\\
			-\operatorname{div}(|x|^{-2a}\nabla v)+\lambda_2\frac{v}{|x|^{2a}}&=\beta q\frac{|u|^p|v|^{q-2}v}{|x|^{b(p+q)}}+\frac{|v|^{2^\sharp-2}v}{|x|^{b{2^\sharp}}},\quad\text{in}~\R^N,\\
			\displaystyle\int_{\R^N}\frac{|u|^2}{|x|^{2a}} dx=\rho_1^2,\quad &\int_{\R^N}\frac{|v|^2}{|x|^{2a}}dx=\rho_2^2.
		\end{aligned}
		\right.
		$$
    where $N\ge3$, $\beta\in\R$, $\rho_1,\rho_2>0$, $\max\{0,\tfrac{N-4}{2}\}\le a<\tfrac{N-2}{2}$, $a<b<a+1$, $2^{\sharp}:=\frac{2N}{N-2(1+a-b)}$, $p,q>1$ and $2<p+q<2^{\sharp}$. Here $2^\sharp$ is the critical exponent of the Caffarelli--Kohn--Nirenberg inequality, and $\lambda_1,\lambda_2\in\R$ are unknown, since the two weighted masses are prescribed. We show that the ground state level is not attained for $\beta\le0$. For $\beta>0$ we obtain positive normalized ground states, with positive Lagrange multipliers, in the mass subcritical, mass critical, and mass supercritical regimes, for explicit ranges of $\beta$. In the mass supercritical regime, a threshold $\beta_0\ge0$ appears: a ground state exists for every $\beta>\beta_0$ and does not exist for $0<\beta<\beta_0$.  We give conditions on $p,q$ under which $\beta_0=0$, and we prove that $\beta_0>0$ when $p,q\ge2$. Finally, we describe the ground states as $\beta\to0^+$ and as $\beta\to+\infty$: after a dilation, they converge to the ground states of the limit system without critical terms, or one component vanishes and the other concentrates on an extremal of the Caffarelli--Kohn--Nirenberg inequality.

\end{abstract}
	\section{Introduction}
	
		In this paper, we study the  following coupled system
		\begin{equation}\label{d400}
			\begin{cases}
				i|x|^{-2a}\partial_t\Psi_1+\operatorname{div}(|x|^{-2a}\nabla\Psi_1)+\beta p\dfrac{|\Psi_1|^{p-2}\Psi_1|\Psi_2|^{q}}{|x|^{b(p+q)}}+\dfrac{|\Psi_1|^{2^\sharp-2}\Psi_1}{|x|^{b2^\sharp}}=0,\\[2mm]
				i|x|^{-2a}\partial_t\Psi_2+\operatorname{div}(|x|^{-2a}\nabla\Psi_2)+\beta q\dfrac{|\Psi_1|^{p}|\Psi_2|^{q-2}\Psi_2}{|x|^{b(p+q)}}+\dfrac{|\Psi_2|^{2^\sharp-2}\Psi_2}{|x|^{b2^\sharp}}=0,
			\end{cases}
		\end{equation}
		for $(x,t)\in\mathbb{R}^N\times\mathbb{R}$, where $N\ge3$, $\Psi_1,\Psi_2:\mathbb{R}^N\times\mathbb{R}\to\mathbb{C}$ are the unknowns, $\beta\in\mathbb{R}$ is the coupling parameter, $p,q>1$, and the exponents and weights satisfy
		\begin{equation}\label{d401}
			\begin{gathered}
				\max\Big\{0,\frac{N-4}{2}\Big\}\le a<\frac{N-2}{2},\qquad a<b<a+1,\qquad d:=1+a-b\in(0,1),\\
				2^\sharp:=\frac{2N}{N-2d},\qquad 2<p+q<2^\sharp .
			\end{gathered}
		\end{equation}
		When $a=b=0$, so that $d=1$ and $2^\sharp=2^*=\frac{2N}{N-2}$, system \eqref{d400} is the classical coupled nonlinear Schr\"odinger system. Such systems arise as mean-field models for binary mixtures of Bose--Einstein condensates and in nonlinear optics; see \cite{bartsch2016normalized,bartsch2021normalized2,bartsch2023existence} and the references therein. In that setting $\beta>0$ describes an attractive interaction between the two components and $\beta<0$ a repulsive one.

		Among all solutions of \eqref{d400}, standing waves are of special interest. The ansatz $\Psi_1(x,t)=e^{i\lambda_1t}u(x)$, $\Psi_2(x,t)=e^{i\lambda_2t}v(x)$ leads to the stationary system
		\begin{equation}\label{d402}
			\begin{cases}
				-\operatorname{div}(|x|^{-2a}\nabla u)+\lambda_1\dfrac{u}{|x|^{2a}}=\beta p\dfrac{|u|^{p-2}u|v|^{q}}{|x|^{b(p+q)}}+\dfrac{|u|^{2^\sharp-2}u}{|x|^{b2^\sharp}},&\text{in }\mathbb{R}^N,\\[2mm]
				-\operatorname{div}(|x|^{-2a}\nabla v)+\lambda_2\dfrac{v}{|x|^{2a}}=\beta q\dfrac{|u|^{p}|v|^{q-2}v}{|x|^{b(p+q)}}+\dfrac{|v|^{2^\sharp-2}v}{|x|^{b2^\sharp}},&\text{in }\mathbb{R}^N,
			\end{cases}
		\end{equation}
		with $\lambda_1,\lambda_2\in\mathbb{R}$. The system \eqref{d402} can be studied in two ways. In the first, $\lambda_1$ and $\lambda_2$ are fixed and the masses of $u$ and $v$ are free. Scalar equations of this type, with singular weights of Caffarelli--Kohn--Nirenberg type, have been studied by many authors; see \cite{wang2000singular,ghoussoub2000multiple,felli2003perturbation,xuan2005solvability} and the references therein. In the second, the masses are prescribed and $\lambda_1,\lambda_2$ are unknown. We follow the second way.
		
		In the past decades, after the work of Jeanjean \cite{jeanjean1997existence}, many researchers have studied solutions of the scalar equation
		\begin{align*}  
			-\Delta u+\lambda u=g(u)\quad\text{in }\mathbb{R}^N,\qquad\int_{\mathbb{R}^N}|u|^2\dx=\rho^2,
		\end{align*}
		with prescribed mass. When $g$ is mass-supercritical the energy is unbounded from below on the $L^2$-sphere, and Jeanjean obtained a solution by a mountain-pass argument on the constraint; see also \cite{bartsch2013normalized} for infinitely many solutions and \cite{jeanjean2020mass} for general mass-supercritical nonlinearities. For the combined nonlinearity $g(u)=\mu|u|^{q-2}u+|u|^{p-2}u$, $2<q<p\le2^*$, Soave \cite{soave2020normalized1,soave2020normalized2} decomposed the Pohozaev manifold according to the sign of the second derivative of the fibering map. This gives a local minimizer when $q$ is mass-subcritical and a mountain-pass solution when $q$ is mass-supercritical, both for $p<2^*$ and for the Sobolev-critical case $p=2^*$, which is by now seen as the Brezis--Nirenberg problem for normalized solutions; the remaining thresholds of the critical case were settled by Wei and Wu \cite{wei2022normalized}. The same approach has since been applied to Choquard, Kirchhoff, fractional, biharmonic and Schr\"odinger--Poisson type equations, and in several of these problems two solutions, a local minimizer and a mountain-pass solution, were obtained; see \cite{li2023nonexistence,zhang2022normalized,liu2024normalized,meng2024normalized,yu2023normalized,jeanjean2021multiple,chen2023multiple,kang2025multiple,jin2025normalized} and the references therein.
		
		For systems, fewer results are available, and the two prescribed masses make the constraint an $L^2$-torus. Bartsch, Jeanjean and Soave \cite{bartsch2016normalized} proved the existence of positive normalized solutions of the cubic system
		\begin{align*}
\left\{
\begin{aligned}
-\Delta u+\lambda_1u&=\mu_1u^3+\beta uv^2,&&\text{in }\mathbb{R}^3,\\
-\Delta v+\lambda_2v&=\mu_2v^3+\beta u^2v,&&\text{in }\mathbb{R}^3.
\end{aligned}
\right.
\end{align*}
		with prescribed masses, for $\beta>0$ in suitable ranges. Bartsch and Soave \cite{bartsch2017natural} introduced a natural constraint of Nehari--Pohozaev type, on which constrained critical points are true solutions. The same constraint was later used to prove multiplicity in the competing case $\beta<0$ \cite{bartsch2019multiple}. Bartsch, Zhong and Zou \cite{bartsch2021normalized2} obtained positive solutions of the cubic system for further ranges of $\beta$. In the fully mass-subcritical case, Gou and Jeanjean \cite{gou2016existence} obtained global minimizers and proved orbital stability, and the compactness of minimizing sequences under several mass constraints was established by Ikoma \cite{Ikoma+2014+115+136}. Li and Zou \cite{li2021normalized} obtained positive normalized ground states for systems with subcritical or critical self-interaction terms and a subcritical coupling by working on the Pohozaev manifold. For systems with Sobolev-critical terms, Bartsch, Li, and Zou \cite{bartsch2023existence} studied
        \begin{equation*}
\begin{cases}
-\Delta u + \lambda_1 u = |u|^{2^*-2}u + \beta p |u|^{p-2}u |v|^q, & \text{in } \mathbb{R}^N, \\
-\Delta v + \lambda_2 v = |v|^{2^*-2}v + \beta q |u|^p |v|^{q-2}v, & \text{in } \mathbb{R}^N, \\
\displaystyle\int_{\mathbb{R}^N} |u|^2 \dx = \rho_1^2, \quad \int_{\mathbb{R}^N} |v|^2 \dx = \rho_2^2,
\end{cases}
\end{equation*}

		with $N=3,4$, $p,q>1$ and $2<p+q<2^*$, and proved nonexistence of ground states for $\beta<0$ and existence, together with the asymptotic behavior of the ground states, for $\beta>0$. The case in which the coupling term itself is Sobolev-critical was studied by Zhang and Han \cite{zhang2024normalized} and, with an additional perturbation ranging over the three mass regimes, by Meng, He, and Winkert \cite{MENG2025113845}.
		
		All the results quoted so far concern the Laplacian, that is, the case $a=b=0$. When $a\ne0$ or $b\ne0$ the operator $-\operatorname{div}(|x|^{-2a}\nabla\,\cdot)$ and the nonlinearities carry weights that are singular at the origin, and the natural functional setting is given by the Caffarelli--Kohn--Nirenberg inequalities \cite{caffarelli1984first}; sharp forms with general weights were obtained more recently by Lam \cite{lam2019general}. Two features of the weighted setting are essential for us. First, the problem is no longer invariant under translations, but it is invariant under the dilations centered at the origin, and the embedding of the weighted Sobolev space into the weighted Lebesgue space with the critical exponent $2^\sharp$ is not compact. Second, the best constant $\mathcal S(a,b)$ of the critical Caffarelli--Kohn--Nirenberg inequality and its extremals are known explicitly: the constant was computed and the extremals were identified by Chou and Chu \cite{chou1993best}, Horiuchi \cite{horiuchi1997best} and Catrina and Wang \cite{catrina2001caffarelli}, and for $a\ge0$, which is our range, the extremals are radial. The complete picture of symmetry and symmetry breaking was obtained by Dolbeault, Esteban and Loss \cite{dolbeault2016rigidity}, and a short proof of the symmetry of extremals is due to Bouchez and Willem \cite{bouchez2009extremal}. These explicit extremals are what make an energy comparison possible in the critical weighted problem.
        Normalized solutions for the weighted operator $-\operatorname{div}(|x|^{-2a}\nabla\,\cdot)$ have been studied only recently. In \cite{goel2026normalized} we considered the scalar problem \begin{equation}\label{d405}
           -\operatorname{div}(|x|^{-2a}\nabla u)+\lambda\frac{u}{|x|^{2a}}
            =\beta\frac{|u|^{q-2}u}{|x|^{bq}}+\frac{|u|^{2^\sharp-2}u}{|x|^{b2^\sharp}}
             \quad\text{in }\mathbb{R}^N,
        \end{equation}
with $0<a<\frac{N-2}{2}$, $a<b<a+1$ and $2<q<2^\sharp$, and proved the existence of a ground
state and of a second constrained critical point in the mass-subcritical case, and of a
mountain-pass ground state in the mass-critical and mass-supercritical cases. Problem
\eqref{d405} is the scalar counterpart of the system treated here.
		
		To the best of our knowledge, no paper considers normalized solutions of a coupled system with the weighted operator $-\operatorname{div}(|x|^{-2a}\nabla\,\cdot)$, Caffarelli--Kohn--Nirenberg critical growth and two prescribed masses. The unweighted system was treated in \cite{bartsch2023existence} for $N=3,4$. Motivated by these works, in this paper we study, for given $\rho_1,\rho_2>0$, the coupled singular weighted system
		\begin{equation}\label{d1}
			\begin{cases}
				-\operatorname{div}(|x|^{-2a}\nabla u)+\lambda_1\dfrac{u}{|x|^{2a}}=\beta p\dfrac{|u|^{p-2}u|v|^{q}}{|x|^{b(p+q)}}+\dfrac{|u|^{2^\sharp-2}u}{|x|^{b2^\sharp}},&\text{in }\mathbb{R}^N,\\[2mm]
				-\operatorname{div}(|x|^{-2a}\nabla v)+\lambda_2\dfrac{v}{|x|^{2a}}=\beta q\dfrac{|u|^{p}|v|^{q-2}v}{|x|^{b(p+q)}}+\dfrac{|v|^{2^\sharp-2}v}{|x|^{b2^\sharp}},&\text{in }\mathbb{R}^N,\\[2mm]
				\displaystyle\int_{\mathbb{R}^N}\frac{|u|^2}{|x|^{2a}}\dx=\rho_1^2,\qquad\int_{\mathbb{R}^N}\frac{|v|^2}{|x|^{2a}}\dx=\rho_2^2 ,
			\end{cases}
		\end{equation}
		where $a,b,d,2^\sharp,p,q$ satisfy \eqref{d401}, $\beta\in\mathbb{R}$, and $\lambda_1,\lambda_2\in\mathbb{R}$ are unknown. In this paper, we prove the existence and nonexistence of normalized ground states for \eqref{d1}, according to the size of $\beta$ and of $p+q$, and we describe their behavior as $\beta\to0^+$ and as $\beta\to+\infty$. All our results hold in every dimension $N\ge3$.
        
		From a mathematical point of view, several difficulties arise. First, $\lambda_1$ and
        $\lambda_2$ are not known in advance. Second, the embedding into the weighted Lebesgue space
with the critical exponent is not compact, and compactness is recovered only for energies below
an explicit threshold given by $\mathcal S(a,b)$, after a comparison with the extremals of the
Caffarelli--Kohn--Nirenberg inequality. Third, the only subcritical term in \eqref{d1} is the
coupling term. There are no subcritical self-interaction terms, so when $\beta$ is small nothing
prevents one component from vanishing, and the limit of a Palais--Smale sequence may be semi-trivial. Ruling this out uses the Pohozaev identity together with the fact that the extremals of the Caffarelli--Kohn--Nirenberg inequality do not have finite weighted mass. Finally, the sign of the multipliers comes from a Liouville type lemma for the weighted
operator, which forces the restriction $a\ge\frac{N-4}{2}$ in \eqref{d401}.

Our main results are stated in Section~\ref{SD1}; we describe them briefly here. For $\beta\le0$
the ground state level is never attained (Theorem~\ref{DD1}), so we concentrate on $\beta>0$.
In each of the three cases we obtain a positive normalized ground state with positive Lagrange
multipliers: as a local minimizer at a negative level in the mass-subcritical case
(Theorem~\ref{DD2}), and as a minimax level in the mass-critical case (Theorem~\ref{DD9}) and in
the mass-supercritical case (Theorem~\ref{DD6}). In the first two cases $\beta$ is below an
explicit constant. In the mass-supercritical case a threshold $\beta_0$ appears instead: a
ground state exists for $\beta>\beta_0$ and does not exist for $0<\beta<\beta_0$, with
$\beta_0=0$ when $\min\{p,q\}<2$ and $p+q$ lies in an explicit range, and $\beta_0>0$ when
$p,q\ge2$. This dichotomy has no counterpart in the scalar problem \eqref{d405}; it comes from
the interaction of the two components near the origin. We then study the limit system obtained
by dropping the critical terms (Theorem~\ref{DD4}), and the behavior of the ground states as
$\beta\to0^+$ and $\beta\to+\infty$. After a suitable dilation they converge to the ground
states of the limit system (Lemmas~\ref{D51} and~\ref{D52}, Theorem~\ref{DD10}), except as
$\beta\to0^+$ in the mass-critical and mass-supercritical cases, where one component vanishes
and the other concentrates on an extremal of the Caffarelli--Kohn--Nirenberg inequality
(Lemma~\ref{D53}). The ideas of the proofs are described at the end of Section~\ref{SD1}.
		
		Throughout the paper, $\|\cdot\|_q$ denotes the norm in $L^q(\mathbb{R}^N)$, so that $\||x|^{-b}u\|_q$ is the norm in $L^q(\mathbb{R}^N;|x|^{-bq})$, and $C$ denotes a positive constant that may change from line to line.
		
		 The paper is organized as follows. Section~\ref{SD1} collects the functional setting, the Caffarelli--Kohn--Nirenberg inequality, the Pohozaev manifold, the statements of the main results, and the ideas of the proofs. Section~\ref{SD2} treats the attractive case $\beta>0$: it contains the compactness result, Proposition~\ref{D5}, and the proofs of Theorems~\ref{DD2}, \ref{DD6} and \ref{DD9} in the three mass regimes. Section~\ref{SD3} proves Theorem~\ref{DD1} for $\beta\le0$. Section~\ref{SD4} studies the limit system and proves Theorem~\ref{DD4}. Section~\ref{SD5} contains the asymptotic analysis as $\beta\to0^+$ and $\beta\to+\infty$. The Liouville-type lemma is proved in the appendix.

	\section{Functional setting and main results}\label{SD1}
	In this section, we introduce the functional framework and the variational structure associated with the problem.
	
	For $2 \le q < \infty$, we denote by $L^q(\R^N;|x|^{-bq})$ the weighted Lebesgue space
	\begin{align*}  
		L^q(\R^N;|x|^{-bq})
		:=\Bigl\{u:\R^N\to\R \text{ measurable } :
		\int_{\R^N}|x|^{-bq}|u|^q\dx<\infty \Bigr\},
	\end{align*}
	endowed with the norm
	\begin{align*}  
		\|u\|_{L^q(\R^N;|x|^{-bq})}
		:=\Bigl(\int_{\R^N}|x|^{-bq}|u|^q\dx\Bigr)^{1/q}.
	\end{align*}
	Let $\mathcal{D}_{a}^{1,2}(\R^N)$ be the completion of $C_0^\infty(\R^N)$ with respect to the norm
	\begin{align*}  
		\|u\|_a
		:=\Bigl(\int_{\R^N}|x|^{-2a}|\nabla u|^2\dx\Bigr)^{1/2}.
	\end{align*}
	We denote by $\mathcal{D}_{a,r}^{1,2}(\R^N)$ the radial subspace.
	Define
	\begin{align*}  
		X
		:=\mathcal{D}_{a,r}^{1,2}(\R^N)\cap L_r^2(\R^N;|x|^{-2a}),
		\qquad
		\mathcal{X}:=X\times X.
	\end{align*}
	The space $\y$ will serve as the natural variational framework for the system.
	
	A key ingredient in our analysis is the Caffarelli–Kohn–Nirenberg inequality \cite{caffarelli1984first}, which ensures that there exists a positive constant $C_q = C_{q}(N, r, p, q, \gamma, \alpha, \eta,\sigma,\delta)$ such that for all $u \in C_0^\infty (\mathbb{R}^N) $,
	\begin{equation*}
		\| |x|^\gamma u \|_q \leq C_{q}\, \| |x|^\alpha \nabla u \|_p^\delta ~\| |x|^\eta u \|_r^{1-\delta}
	\end{equation*}
	where $p,~ r \geq 1,~ q > 0,~ 0 \leq \delta \leq 1 $, $\sigma\in\R$, and 
	$\frac{1}{p} + \frac{\alpha}{N}, ~\frac{1}{r} + \frac{\eta}{N} , ~\frac{1}{q} + \frac{\gamma}{N} > 0$
	where
	\begin{equation*}
		\gamma = \delta\sigma + (1 - \delta) \eta
	\end{equation*}
	and
	\begin{equation*}
		\frac{1}{q} + \frac{\gamma}{N} = \delta\left( \frac{1}{p} + \frac{\alpha-1}{N} \right) + (1 - \delta) \left( \frac{1}{r} + \frac{\eta}{N} \right).
	\end{equation*}
	Also,
	\begin{equation*}
		0 \leq \alpha - \sigma \quad \text{if} \quad \delta> 0,\quad\text{and}
	\end{equation*}
	\begin{equation*}
		\alpha - \sigma \leq 1 \quad \text{if} \quad \delta > 0, \quad \text{and} \quad \frac{1}{p} + \frac{\alpha - 1}{N} = \frac{1}{q} + \frac{\gamma}{N}.
	\end{equation*}
	Moreover, $C_{q}$ is bounded if and only if $(\alpha-\sigma)\in[0,1].$

	For our variational analysis, we require a Gagliardo--Nirenberg type inequality
	involving weighted $L^2$ norms.
	\begin{lemma}\cite{goel2026normalized}\label{lem:GN}
		Let $u\in \mathcal{D}_{a,r}^{1,2}(\R^N)$ and $2< q < 2^{\sharp}$.
		There exists a constant $C_q>0$ such that
		\begin{equation*}
			\||x|^{-b}u\|_q^q
			\le C_q^q
			\||x|^{-a}\nabla u\|_2^{\delta_q q}
			\||x|^{-a}u\|_2^{(1-\delta_q)q},
		\end{equation*}
		where
		\begin{align*}  
			\delta_q=\frac{(N-2a+2b)q-2N}{2q}.
		\end{align*}
	\end{lemma}
	Denote
	\begin{align*}  
		\mathcal{B}_{p,q}(u,v)
		:=\int_{\R^N}\frac{|u|^p|v|^q}{|x|^{b(p+q)}}\dx.
	\end{align*}
	By H\"older’s inequality and Lemma~\ref{lem:GN}, we have
	\begin{align}\label{d5}
		\mathcal{B}_{p,q}(u,v)
		&\le \||x|^{-b}u\|_{p+q}^p\,\||x|^{-b}v\|_{p+q}^q \notag\\
		&\le C_{a,b}^{p+q}
		\Bigl(\||x|^{-a}\nabla u\|_2^2+\||x|^{-a}\nabla v\|_2^2\Bigr)^{\frac{(p+q)\delta_{p+q}}{2}}
		\Bigl(\||x|^{-a}u\|_2^2+\||x|^{-a}v\|_2^2\Bigr)^{\frac{(p+q)(1-\delta_{p+q})}{2}}.
	\end{align} 
	The mass-critical exponent is given by
	\begin{align*}  (p+q)_c:=\frac{2N+4}{N-2a+2b}.\end{align*}
	\begin{prop}\cite{goel2026normalized}\label{D303}
		Let $ N \geq 3 $, $ 0 < a < \frac{N-2}{2} $, and $ a < b < 1 + a $. Then, for any $ q  \in (2, \t) $ the embedding
		\begin{align*}  
			X \hookrightarrow L^q_r(\mathbb{R}^N; |x|^{-qb} )
		\end{align*}
		is compact.
	\end{prop}
	Since $X$ is not compactly embedded in 
	$L_r^{\t}(\mathbb{R}^N; |x|^{-\t b})$, 
	loss of compactness occurs at the critical exponent. 
	In the critical case $q=\t$, Catrina and Wang \cite{catrina2001caffarelli} 
	studied the extremals associated with the embedding and introduced the best 
	Caffarelli--Kohn--Nirenberg constant (see also \cite{bouchez2009extremal})
	\begin{equation}\label{d191}
		\mathcal{S}(a,b)
		:=
		\inf_{u\in\mathcal{D}_a^{1,2}(\mathbb{R}^N)\setminus\{0\}}
		\frac{\||x|^{-a}\nabla u\|_2^2}
		{\||x|^{-b}u\|_\t^{2}}.
	\end{equation}
	
	In \cite{catrina2001caffarelli}, it was proved that the minimizers of 
	$\mathcal{S}(a,b)$ belong to the family $\{U_\varepsilon\}_{\varepsilon>0}$, 
	given by
	\begin{equation}\label{c96}
		U_{\varepsilon}(x)
		=
		\frac{\left(2\t A\varepsilon\right)^{\frac{N-2d}{4d}}}
		{\left(\varepsilon+|x|^{\alpha}\right)^{\frac{N-2d}{2d}}},
	\end{equation}
	where
	\begin{align*}  
		A=\left(\frac{N-2}{2}-a\right)^2,
		\qquad
		\alpha=\frac{2d(N-2-2a)}{N-2d}.
	\end{align*}
	The system under consideration \eqref{d1} admits a variational formulation. More precisely, solutions correspond to critical points of the energy functional
	$E_\beta:\mathcal{X}\to\R$ defined by
	\begin{align*}  
		E_{\beta}(u,v)
		:=\frac{1}{2}\Bigl(\||x|^{-a}\nabla u\|_2^2+\||x|^{-a}\nabla v\|_2^2\Bigr)
		-{\beta}\mathcal{B}_{p,q}(u,v)
		-\frac1{2^{\sharp}}\Bigl(\||x|^{-b}u\|_{2^{\sharp}}^{2^{\sharp}}
		+\||x|^{-b}v\|_{2^{\sharp}}^{2^{\sharp}}\Bigr),
	\end{align*}
	restricted to the mass constraint manifold of $\y$, defined by
	\begin{align*}  S(\rho_1,\rho_2):=\{(u,v)\in\y:\||x|^{-a}u\|_2^2=\rho_1^2,\||x|^{-a}v\|_2^2=\rho_2^2\}.\end{align*}
	Since $p,q>1$, $E_\beta\in C^1(\y,\R)$. The Pohozaev functional associated with \eqref{d1} is
	\begin{align*}  
		P_{\beta}(u,v)
		&:=\||x|^{-a}\nabla u\|_2^2+\||x|^{-a}\nabla v\|_2^2
		-\beta\pq\mathcal{B}_{p,q}(u,v)
		-\Bigl(\||x|^{-b}u\|_{2^{\sharp}}^{2^{\sharp}}
		+\||x|^{-b}v\|_{2^{\sharp}}^{2^{\sharp}}\Bigr).
	\end{align*}
	Let $\u\in S\r$ be a constrained critical point of $E_\beta$, i.e., $\u$
	is a critical point of $E_\beta$ restricted to the torus $S\r$. Then $\u$ belongs to the Pohozaev
	manifold $\mathcal{N}_{\beta}(\rho_1,\rho_2)$, where 
	\begin{align*}  
		\mathcal{N}_{\beta}(\rho_1,\rho_2)
		:=\{(u,v)\in S(\rho_1,\rho_2): P_{\beta}(u,v)=0\}.
	\end{align*} 
	To study the geometry of the functional, we introduce the mass-preserving scaling
	\begin{align*}  k\star u(x):=e^{\frac{N-2a}{2}k}u(e^kx)\quad\text{and}\quad k\star(u,v):=(k\star u,k\star v),\end{align*}
	which leaves the constraint manifold invariant. For  $(u,v)\in S\r$, we define the associated fiber map, $\Phi_\beta^{\u}:\R\to\R$,  by
	\begin{align*}  
		\Phi_\beta^\u(t):=E_\beta(t\star\u),
	\end{align*}
	i.e.,
	\begin{align*}  
		\Phi_{\beta}^\u(t)= \frac{e^{2t}}{2}(\||x|^{-a}\na u\|_2^2+\||x|^{-a}\na v\|_2^2)-\beta e^{(p+q)\delta_{p+q}t}\b(u,v)-\frac{e^{\t t}}{\t}(\||x|^{-b}u\|_{\t}^{\t}+\||x|^{-b}v\|_\t^\t).
	\end{align*}
	A direct computation shows that $(\Phi_{\beta}^\u)'(0)=P_\beta\u$ so that critical points of the fiber map correspond to intersections with the Pohozaev manifold. We can decompose the Pohozaev manifold $\mathcal{N}_\beta\r$ as
	\begin{equation*}
		\mathcal{N}_{\beta}\r=\mathcal{N}_{\beta}^+\r\cup\mathcal{N}_{\beta}^-\r\cup\mathcal{N}_{\beta}^0\r,
	\end{equation*}
	where
	\begin{align*}  
		\mathcal{N}_{\beta}^+\r:=\{\u\in \mathcal{N}_{\beta}\r:
		(\Phi_{\beta}^\u)''(0)>0\},\\
		\mathcal{N}_{\beta}^-\r:=\{\u\in \mathcal{N}_{\beta}\r:(\Phi^\u_{\beta})''(0)<0\},\\
		\mathcal{N}_{\beta}^0\r:=\{\u\in \mathcal{N}_{\beta}\r:(\Phi^\u_{\beta})''(0)=0\}.
	\end{align*}
	
	This decomposition plays a fundamental role in the variational analysis, as it reflects the behavior of the energy functional under mass-preserving dilations.
	The ground state level is
	\begin{align*}  m_\beta\r:=\inf_{\mathcal{N}_\beta\r}E_\beta\u.\end{align*}
	
	\begin{theorem}[Mass subcritical]\label{DD2}
		If $N\ge 3, ~\max\{0,\frac{N-4}{2}\}\le a<\frac{N-2}{2},$ $a<b<a+1,~1<p,~q<\t,$ and  $(p+q)<\p$ then for any $0<\beta<\beta_1$ (see \eqref{d200}), there exists a normalized ground state solution $(u_\beta,v_\beta)$ of \eqref{d1}, with $u_\beta,v_\beta>0$ and $\lambda_1,\lambda_2>0$. Furthermore, this solution is a local minimizer of the energy functional $E_\beta$ on $S\r.$ 
	\end{theorem}
	\begin{theorem}[Mass supercritical]\label{DD6}
		If $N\ge 3, ~\max\{0,\frac{N-4}{2}\}\le a<\frac{N-2}{2},$ $a<b<a+1,~1<p,~q<\t,$ and  $(p+q)>\p$, then there exists $\beta_0\ge0$ such that a normalized ground state solution exists  for every $\beta>\beta_0$ and does not exist for $0<\beta<\beta_0$.  Moreover, $\beta_0=0$ if one of the following holds: \begin{itemize}
			\item[(i)]  $a=\frac{N-4}{2}$, with $N>4$.
			\item[(ii)]  $\max\{0,\frac{N-4}{2}\}<a<\frac{N-2}{2}$ ,~$r:=\min\{p,q\}<2$ and
			\begin{align*}  \max\Big\{\p,\frac{N}{N-2-2a+b},\t-\frac{2(2-r)(N-2a-2)}{4+2a-N}\Big\}<p+q<\t,\end{align*}
			\item[(iii)] $0<\frac{N-4}{2}<a<\frac{N-2}{2}$, $r<2$ and
			\begin{align*}  \p<p+q<\min\Big\{\frac{N}{N-2-2a+b},\t-\frac{2(r-1)}{N-2(1+a-b)}\Big\}\end{align*}
			\item[(iv)] $\max\{0,\frac{N-4}{2}\}<a<\frac{N-2}{2}$, $r<2$ and $\p<p+q=\frac{N}{N-2-2a+b}.$ 
		\end{itemize}          While $\beta_0>0$
		if $p,~q\ge2$. Moreover, the corresponding Lagrange multipliers $\lambda_1,~\lambda_2>0.$ 
	\end{theorem}
	\begin{theorem}[Mass critical]\label{DD9}
		If $N\ge 3, ~\max\{0,\frac{N-4}{2}\}\le a<\frac{N-2}{2},$ $a<b<a+1,~1<p,~q<\t,$ $(p+q)=\p$ and $0<\beta<\beta^*$ (see \eqref{d222}), one of the conditions (i)--(iii) of Lemma~\ref{D93} being satisfied, then there exists a normalized ground state solution $(u_\beta,~v_\beta)$ with positive Lagrange multipliers $\lambda_1,\lambda_2$.
		
	\end{theorem}

	\begin{theorem}\label{DD1} If $N\ge 3, ~\max\{0,\frac{N-4}{2}\}\le a<\frac{N-2}{2},~a<b<a+1$ and $\beta\le 0$ then $m_\beta\r$ is not achieved.    
	\end{theorem}
	
	In this paper, we are also concerned with the asymptotic behavior of the normalized ground-state solutions obtained in  Theorems~\ref{DD2} and \ref{DD6} as $\beta\to 0^+$ and $\beta\to +\infty$. Similar results have been obtained for the case $a=b=0$ in \cite{bartsch2023existence}. The limit system is
	\begin{equation}\label{d60}
		\left\{
		\begin{aligned}
			-\text{div}(|x|^{-2a}\na u)&+\lambda_1 \frac{u}{|x|^{2a}}=p\frac{|v|^q|u|^{p-2}u}{|x|^{b(p+q)}},\\
			-\text{div}(|x|^{-2a}\na v)&+\lambda_2\frac{v}{|x|^{2a}}=q\frac{|u|^p|v|^{q-2}v}{|x|^{b(p+q)}},\\
			\||x|^{-a}u\|_2^2=\rho^2_1,&\quad \||x|^{-a}v\|_2^2=\rho_2^2.
		\end{aligned}
		\right.
	\end{equation}
	\begin{theorem}\label{DD4}
		Suppose $N\ge 3,~\max\{0,\frac{N-4}{2}\}\le a<\frac{N-2}{2},~a<b<1+a$, and $p+q\ne \p$. Then $\kappa\r$ (see \eqref{d61}) is achieved by a positive normalized ground state of \eqref{d60}, and the corresponding Lagrange multipliers satisfy $\lambda_1,\lambda_2>0$.
	\end{theorem}
    \begin{theorem}\label{DD11}
Let $N \ge 3$, $\max\left\{0, \frac{N-4}{2}\right\} \le a < \frac{N-2}{2}$, $a < b < a+1$, and let $(u_\beta, v_\beta) \in \mathcal{N}_\beta(\rho_1, \rho_2)$ be a family of normalized ground states of system \eqref{d1} attaining $m_\beta(\rho_1, \rho_2)$.
\begin{itemize}
    
    \item[(i)] \emph{(Mass-subcritical regime)} If $p+q < (p+q)_c$, let $t_\beta := t_\beta^{(u_0, v_0)}$ be given by Lemma~\ref{D3} for any $(u_0, v_0) \in \mathcal{Z}(\rho_1, \rho_2)$. Then $e^{t_\beta} \sim \beta^{\frac{1}{2 - (p+q)\delta_{p+q}}}$ as $\beta \to 0^+$, and
    \[
    \operatorname{dist}_{\mathcal{X}}\big((-t_\beta)*(u_\beta, v_\beta), \, \mathcal{Z}(\rho_1, \rho_2)\big) \longrightarrow 0 \quad \text{as } \beta \to 0^+,
    \]
    where $\mathcal{Z}(\rho_1, \rho_2)$ denotes the set of normalized ground states of the limit system \eqref{d60}. Moreover,
    \[
    \lambda_{1,\beta} + \lambda_{2,\beta} \sim \||x|^{-a}\nabla u_\beta\|_2^2 + \||x|^{-a}\nabla v_\beta\|_2^2 \sim \beta^{\frac{2}{2 - (p+q)\delta_{p+q}}} \quad \text{as } \beta \to 0^+.
    \]

    \item[(ii)] \emph{(Mass-critical and mass-supercritical regimes)} If $p+q \ge (p+q)_c$, then as $\beta \to 0^+$, up to a subsequence, one component vanishes while the other concentrates on an extremal of the Caffarelli--Kohn--Nirenberg inequality; that is,
    \[
    \operatorname{dist}_{\mathcal{D}_a^{1,2} \times \mathcal{D}_a^{1,2}}\big((u_\beta, v_\beta), \, \mathcal{U} \times \{0\}\big) \longrightarrow 0 \quad \text{or} \quad \operatorname{dist}_{\mathcal{D}_a^{1,2} \times \mathcal{D}_a^{1,2}}\big((u_\beta, v_\beta), \, \{0\} \times \mathcal{U}\big) \longrightarrow 0,
    \]
    where $\mathcal{U} = \{U_\epsilon : \epsilon > 0\}$ is the family of extremals defined in \eqref{c96}.
\end{itemize}
\end{theorem}
	\begin{theorem}\label{DD10}
		Let $N \ge 3$, $\max\left\{0, \frac{N-4}{2}\right\} \le a < \frac{N-2}{2}$, $a < b < a+1$, $\t = \frac{2N}{N-2(1+a-b)}$, and $(p+q) > (p+q)_c $. Let $(u_\beta, v_\beta) \in \mathcal{N}_\beta(\rho_1, \rho_2)$ be a family of normalized ground states of system \eqref{d1} attaining $m_\beta(\rho_1, \rho_2)$. Then, as $\beta \to +\infty$:
		\begin{align*}  
			\operatorname{dist}_{\mathcal{X}}\left( s_\beta \star (u_\beta, v_\beta), \, \mathcal{Z}(\rho_1, \rho_2) \right) \longrightarrow 0
		\end{align*}
		where $\mathcal{Z}(\rho_1, \rho_2)$ is the set of normalized ground states of the limit system, and the $L^2(\R^N;| x|^{-2a})$-invariant scaling parameter is given by:
		\begin{align*}  
			s_\beta := \frac{1}{(p+q)\delta_{p+q} - 2} \ln \beta.
		\end{align*}
	\end{theorem}

	We conclude with the ideas of the proofs. For $\u\in S\r$ we study the fibering map 	$t\mapsto E_\beta(t\star\u)$. In the mass-subcritical case it has two critical points, a local minimum at a negative level and a global maximum, and this splits $\mathcal N_\beta\r$ into 	$\mathcal N_\beta^+\r$ and $\mathcal N_\beta^-\r$ (Lemma~\ref{D3}). In the other two cases, it has a single critical point, which is a global maximum (Lemmas~\ref{D6} and \ref{D12}). 	Accordingly, $m_\beta\r$ is a local minimum in the first case and a minimax level in the other two (Lemmas~\ref{D7} and \ref{D301}). In each case, we obtain a Palais--Smale sequence at that level with $P_\beta\v\to0$ and $u_n^-,v_n^-\to0$ (Lemmas~\ref{D9} and~\ref{D91}). Such a 	sequence converges strongly in $\y$ provided
	\begin{align*}
		m_\beta\r<\frac dN\sr^{N/2d}+\min\{0,m_\beta\r\}
	\end{align*}
	(Proposition~\ref{D5}). The proof uses the Brezis--Lieb splitting 		\cite{brezis1983relation}, it excludes semi-trivial limits because the extremals $U_\var$ do not have finite weighted mass, and it excludes loss of mass at infinity by the monotonicity of $m_\beta$ in the masses (Lemmas~\ref{D4}, \ref{D11} and \ref{D92}). The energy condition holds 	automatically when $m_\beta\r<0$. Otherwise, we test the level with pairs built from the truncated extremals $\zeta_\var=\phi U_\var$ of \eqref{c96}, whose second component carries a	small mass $\var^\theta$ (Lemma~\ref{D19}). The competition between the loss 		$O(\var^{\frac{N-2d}{2d}})$ and the gain $\var^{q\theta}$ produces the conditions on $p,q,a,b$ 		and the threshold $\beta_0$ (Lemmas~\ref{D201} and \ref{D93}). The multipliers are positive by	a Liouville-type lemma for the weighted operator (Lemma~\ref{D2}), and this is where $a\ge\frac{N-4}{2}$ is used. The asymptotics as $\beta\to0^+$ rest on the comparison of $E_\beta$ with the limit functional $J$ of Section~\ref{SD4}. Those as $\beta\to+\infty$ rest on the dilation $s_\beta\star$, under which the critical term becomes a vanishing perturbation of $J$.
	
	\section{ Ground states for \texorpdfstring{$\beta>0$}{beta > 0}}\label{SD2}
	
	This section investigates the existence of normalized solutions for $\beta > 0$ in all three $(p+q)$ regimes. We begin with key compactness results.
	\begin{prop}\label{D5}
		Assume  $m_\beta(\gamma_1, \gamma_2) \le m_\beta(\rho_1, \rho_2)$ whenever $0 < \rho_1 \le \gamma_1$ and $0 < \rho_2 \le \gamma_2$.
 Let $\{(u_n, v_n)\} \subset S\r$ be a Palais--Smale sequence for $E_\beta\big|_{S\r}$ at level $c \neq 0$ with $P_\beta(u_n, v_n) \to 0$ as $n \to \infty$, and $u_n^- \to 0$, $v_n^- \to 0$ a.e. in $\mathbb{R}^N$. If
		\begin{equation}\label{d2}
			c < \frac{d \mathcal{S}(a,b)^{N/2d}}{N} + \min\{0, m_\beta\r\},
		\end{equation}
		then, up to a subsequence, $(u_n, v_n) \to (u,v)$ in $\mathcal{X}$ with $u, v > 0$ and $\lambda_1, \lambda_2 > 0$, and $(\lambda_{1,n}, \lambda_{2,n}) \to (\lambda_1, \lambda_2)$ in $\mathbb{R}^2$.  
	\end{prop}
	\begin{proof}
		The proof has three steps.
		
		\noindent
		\textbf{Step 1.} We show that $\{(u_n,v_n)\}$ is bounded in $\mathcal{X}$ and  $\{(\lambda_{1,n},\lambda_{2,n})\}$ is bounded in  $\R^2$. Consider first the case $2<p+q<\p$. For  large $n$,  since $P_\beta(u_n,v_n)\to 0$,
		\begin{align*}  
			c+1&\ge E_\beta\v-\frac{1}{\t}P_\beta\v\\
			&=\frac{d}{N}\an-\beta\frac{\t-(p+q)\delta_{p+q}}{\t}\b\v\\
			&\ge \frac{d}{N}\an-C\an^{\frac{(p+q)\delta_{p+q}}{2}},
		\end{align*}
		for some $C>0$, since $\pq<2$.  Hence, $\{(u_n,v_n)\}$ is  bounded in $\mathcal{X}$. Now suppose $\p\le p+q<\t$.
		Again, using  $P_\beta\v\to0$, we deduce     
		\begin{align*}  
			c+1&\ge E_{\beta}(u_n,v_n)-\frac{1}{2}P_\beta(u_n,v_n)\\
			&=\frac{d}{N}(\||x|^{-b}u_n\|_\t^\t+\||x|^{-b}v_n\|_\t^\t)+\beta\frac{(p+q)\delta_{p+q}-2}{2}\b(u_n,v_n)\\
			&\ge C\an +o_n(1)
		\end{align*} 
		with $C>0$, where the last step uses $P_\beta\v\to0$ to bound the gradient term by the two integrals on the previous line. Hence $\{(u_n,v_n)\}$ is bounded in $\mathcal{X}$.  Since $\v$ is a Palais--Smale sequence for $E_\beta|_{S\r}$, we have
		\begin{align*}  \lambda_{1,n}=-\frac{1}{\rho_1^2}E_\beta'\v[(u_n,0)]+o_n(1)\end{align*}
		and 
		\begin{align*}  
			\lambda_{2,n}=-\frac{1}{\rho_2^2}E_\beta'\v[(0,v_n)]+o_n(1).
		\end{align*}
		so $\{(\lambda_{1,n},\lambda_{2,n})\}$  is bounded in $\mathbb{R}^2$.  Thus, there exists $\u\in\y,~(\lambda_1,\lambda_2)\in\R^2$ such that  up to a subsequence
		\begin{align*}  
			\v\rightharpoonup \u\quad\text{in}~\y,
		\end{align*}
		\begin{align*}  
			\v\to\u\quad\text{in}~L^{p+q}_r(\R^N;|x|^{-b(p+q)})\times L^{p+q}_r(\R^N;|x|^{-b(p+q)}),
		\end{align*}
		\begin{align*}  \v\to\u\quad\text{a.e. in}~\R^N,\end{align*}
		\begin{align*}  (\lambda_{1,n},\lambda_{2,n})\to(\lambda_1,\lambda_2)\quad\text{in}~\R^2.\end{align*}
		Since $E'_{\beta}\v|_{S\r}\to0$ and $u_n^-\to 0,~v_n^-\to 0$ a.e., we obtain 
		\begin{equation}\label{d214}
			\begin{cases}
				E'_{\beta}\u+\left(\lambda_1\dfrac{u}{|x|^{2a}},\lambda_2\dfrac{v}{|x|^{2a}}\right)=0,\\
				u\ge0,~v\ge0, 
			\end{cases}
		\end{equation}
		whence $P_{\beta}\u=0.$
        
		\noindent
		\textbf{Step 2.} We prove that $u\not\equiv 0$ and $v\not\equiv 0$.
		
		Assume by contradiction that $v\equiv 0$. Then \eqref{d214} reduces to
		\begin{align*}  
			-\operatorname{div}\bigl(|x|^{-2a}\nabla u\bigr)+\lambda_1\frac {u}{|x|^{2a}}
			=\frac{|u|^{\t-2}u}{|x|^{b\t}},\qquad u\ge 0,\quad u\in X.
		\end{align*}
		The Pohozaev identity yields that either $\lambda_1=0$ or $u\equiv 0$. If $\lambda_1=0$ and $u\not\equiv 0$, then the maximum principle implies $u>0$, and hence $u=U_\var$ for some $\var>0$. This is impossible since $U_\var\notin L^2(\mathbb{R}^N;|x|^{-2a})$. Therefore, $v\equiv0$ forces $u\equiv0$ as well, and by symmetry $u\equiv0$ forces $v\equiv0$; it remains to rule out $u\equiv v\equiv0$.
		Up to a subsequence, we may assume that
		\begin{align*}  
			\||x|^{-a}\nabla u_n\|_2^2\to \ell_1^2\ge 0,
			\qquad
			\||x|^{-a}\nabla v_n\|_2^2\to \ell_2^2\ge 0.
		\end{align*}
		We now show that $\ell_1^2+\ell_2^2>0$. Suppose instead that $\ell_1=\ell_2=0$. Then
		\begin{align*}  
			E_\beta(u_n,v_n)
			=\frac{d}{N}\Bigl(\||x|^{-a}\nabla u_n\|_2^2+\||x|^{-a}\nabla v_n\|_2^2\Bigr)
			+\frac{1}{\t}P_\beta(u_n,v_n)+o_n(1)\to 0,
		\end{align*}
		as $n\to\infty$, contradicting $c\ne 0$.
		Since $P_\beta(u_n,v_n)\to 0$ and $\b\v\to\b(0,0)=0$, we obtain
		\begin{align*}  
			\ell_1^2+\ell_2^2
			&=\lim_{n\to\infty}\Bigl(\||x|^{-a}\nabla u_n\|_2^2+\||x|^{-a}\nabla v_n\|_2^2\Bigr) \\
			&=\lim_{n\to\infty}\Bigl(\||x|^{-b}u_n\|_\t^\t+\||x|^{-b}v_n\|_\t^\t\Bigr) \\
			&\le \mathcal{S}(a,b)^{-\t/2}\bigl(\ell_1^\t+\ell_2^\t\bigr).
		\end{align*}
		Consequently, 
		\begin{align*}  
			c
			&=\lim_{n\to\infty}E_\beta(u_n,v_n)
			=\frac{d}{N}(\ell_1^2+\ell_2^2) \\
			&\ge \frac{d}{N}\min\Bigl\{\eta_1^2+\eta_2^2:
			\eta_1^2+\eta_2^2\le \sr^{-\t/2}(\eta_1^\t+\eta_2^\t),\
			\eta_1+\eta_2>0,\
			\eta_1,\eta_2\ge 0\Bigr\} \\
			&=\frac{d}{N}\min\Bigl\{r^2:
			r^2\le r^\t \sr^{-\t/2}(\cos^\t\theta+\sin^\t\theta),\
			r\ne 0,\
			\theta\in\Bigl[0,\frac{\pi}{2}\Bigr]\Bigr\}
			=\frac{d}{N}\sr^{N/2d},
		\end{align*}
		which contradicts \eqref{d2}. Hence $v\not\equiv 0$, and similarly $u\not\equiv 0$.
        
		\noi
		\textbf{Step 3.}  We show that $\v\to\u$ in $\y$.
		
		Set $(\bar u_n,\bar v_n)=(u_n-u,v_n-v)$. The Brezis--Lieb lemma \cite{brezis1983relation} gives 
		\begin{align*}  
			E_\beta(u_n,v_n)-E_\beta(u,v)-E_\beta(\bar u_n,\bar v_n)=o_n(1).
		\end{align*}
		Together with the weak convergence and $P_\beta(\bar u_n,\bar v_n)=P_\beta(u_n,v_n)-P_\beta(u,v)+o_n(1)=o_n(1)$, this  gives
		\begin{align*}  
			\lim_{n\to\infty}\Bigl(\||x|^{-a}\nabla\bar u_n\|_2^2+\||x|^{-a}\nabla\bar v_n\|_2^2\Bigr)
			=
			\lim_{n\to\infty}\Bigl(\||x|^{-b}\bar u_n\|_\t^\t+\||x|^{-b}\bar v_n\|_\t^\t\Bigr).
		\end{align*}
		By the argument of Step~2, the limit either vanishes or is at least $\sr^{N/2d}$.  If the latter holds, then
		\begin{align*}  
			c
			&=\lim_{n\to\infty}E_\beta(u_n,v_n)
			=E_\beta(u,v)+\lim_{n\to\infty}E_\beta(\bar u_n,\bar v_n) \\
			&\ge m_\beta\bigl(\||x|^{-a}u\|_2,\||x|^{-a}v\|_2\bigr)
			+\lim_{n\to\infty}\frac{d}{N}\Bigl(\||x|^{-a}\nabla\bar u_n\|_2^2+\||x|^{-a}\nabla\bar v_n\|_2^2\Bigr) \\
			&\ge m_\beta\bigl(\||x|^{-a}u\|_2,\||x|^{-a}v\|_2\bigr)+\frac{d}{N}\sr^{N/2d}\\
			& \ge m_\beta\r+\frac{d}{N}\sr^{N/2d}\ge\min\{0,m_\beta\r\}+\frac dN\sr^{N/2d},
		\end{align*}
		where we used that $\u\in\mathcal N_\beta(\||x|^{-a}u\|_2,\||x|^{-a}v\|_2)$, that $\||x|^{-a}u\|_2\le\rho_1$, $\||x|^{-a}v\|_2\le\rho_2$ by weak lower semicontinuity, and the monotonicity assumption on $m_\beta$.
		which contradicts \eqref{d2}. Hence,
		\begin{align*}  
			\lim_{n\to\infty}\Bigl(\||x|^{-a}\nabla\bar u_n\|_2^2+\||x|^{-a}\nabla\bar v_n\|_2^2\Bigr)=0.
		\end{align*}
		
		It remains to show that $\lambda_1,\lambda_2>0$. If $\lambda_1\le 0$, then
		\begin{align*}  
			-\operatorname{div}\bigl(|x|^{-2a}\nabla u\bigr)
			=
			|\lambda_1|\frac{u}{|x|^{2a}}
			+\beta p\,\frac{|v|^q|u|^{p-2}u}{|x|^{b(p+q)}}
			+\frac{|u|^{2^\sharp-2}u}{|x|^{b2^\sharp}}
			\ge 0
			\quad\text{in } \mathbb{R}^N.
		\end{align*}
		By Lemma~\ref{D2}, applicable because $u\in L^2(\R^N;|x|^{-2a})$ and $a\ge\frac{N-4}{2}$, this implies $u\equiv 0$, a contradiction. Therefore $\lambda_1>0$, and similarly $\lambda_2>0$.  Testing the Palais--Smale condition with $\v$, testing \eqref{d214} with $\u$, and using the strong convergence of the gradients from Step 3 together with $\v\to\u$ in the weighted $L^{p+q}$ and $L^{\t}$ norms, we get
		\begin{align*}  
			\lambda_1\rho_1^2+\lambda_2\rho_2^2
			=
			\lambda_1\||x|^{-a}u\|_2^2+\lambda_2\||x|^{-a}v\|_2^2.
		\end{align*}
		Since $\lambda_1,\lambda_2>0$ and $\||x|^{-a}u\|_2\le\rho_1$, $\||x|^{-a}v\|_2\le\rho_2$, both differences $\rho_i^2-\||x|^{-a}\cdot\|_2^2$ must vanish. Hence $\||x|^{-a}u\|_2^2=\rho_1^2$ and $\||x|^{-a}v\|_2^2=\rho_2^2$, and therefore $(u_n,v_n)\to(u,v)$ in $\mathcal{X}$.\qed   
	\end{proof}
	
	\subsection{Mass-subcritical case}
	In this subsection we treat the case $p+q<\p$ and prove Theorem~\ref{DD2}, for $0<\beta<\beta_1$. 
	
	By the definition of $\sr$ and by inequality \eqref{d5}, we obtain 
	\begin{align*}  
		E_\beta\u&\ge\frac{1}{2}\a-\frac{1}{\t\mathcal{S}(a,b)^{\t/2}}\a^{\t/2}\\&\quad-\beta C_{a,b}^{p+q}(\rho_1^2+\rho_2^2)^{\frac{(p+q)(1-\delta_{p+q})}{2}}\a^{\frac{\pq}{2}}.
	\end{align*}
	Define $f:\R^+\to\R$,
	\begin{equation*}
		f(t):=\frac{1}{2}t^2-\beta C_{a,b}^{p+q}(\rho_1^2+\rho_2^2)^{\frac{(p+q)(1-\delta_{p+q})}{2}}t^{\pq}-\frac{1}{\t\mathcal{S}(a,b)^{\t/2}}t^{\t}.
	\end{equation*}
	If $\beta>0$ and $\pq<2$, then  $f(0^+)=0^-$ and $f(+\infty)=-\infty.$
	\begin{lemma}
		The function $f$ has only two critical points, a strict local minimum at a negative level and a strict global maximum at a positive level, for $0<\beta<\beta_1$. Additionally, there exist ${\kappa}_1:={\kappa}_1(\beta,\rho_1,\rho_2)$ and ${\kappa}_2:={\kappa}_2(\beta,\rho_1,\rho_2)$, such that $f({\kappa}_1)=f({\kappa}_2)=0$ and $f(t)>0$ if and only if $t\in({\kappa}_1,{\kappa}_2).$
	\end{lemma}
	\begin{proof}
		We rewrite $f$ as
		$$
		f(t)=t^{\pq}\left(\frac{1}{2}t^{2-(p+q)\delta_{p+q}}-\beta C_{a,b}^{p+q}(\rho_1^2+\rho_2^2)^{\frac{(p+q)(1-\delta_{p+q})}{2}}-\frac{1}{\t\mathcal{S}(a,b)^{\t/2}}t^{\t-(p+q)\delta_{p+q}}\right).
		$$
		Therefore, $f>0$ if and only if
		\begin{equation*}
			h(t)>\beta C_{a,b}^{p+q}(\rho_1^2+\rho_2^2)^{\frac{(p+q)(1-\delta_{p+q})}{2}},
		\end{equation*}
		where
		$$
		h(t)=\frac{1}{2}t^{2-(p+q)\delta_{p+q}}-\frac{1}{\t\mathcal{S}(a,b)^{\t/2}}t^{\t-(p+q)\delta_{p+q}}.
		$$
		A simple calculation gives
		\begin{equation*}
			h'(t)=\frac{2-(p+q)\delta_{p+q}}{2}t^{1-(p+q)\delta_{p+q}}-\frac{\t-(p+q)\delta_{p+q}}{\t\mathcal{S}(a,b)^{\t/2}}t^{\t-1-(p+q)\delta_{p+q}}.
		\end{equation*}
		The equation $h'(t)=0$ has a unique solution
		\begin{equation}\label{d9}
			{t}^*:=\left(\frac{\t\mathcal{S}(a,b)^{\t/2}(2-(p+q)\delta_{p+q})}{2(\t-(p+q)\delta_{p+q})}\right)^{\frac{1}{\t-2}}.\end{equation}
		Thus, $h$ is increasing on $(0,t^*)$, decreasing on $(t^*,+\infty)$, and achieves its unique global maximum at $t=t^*$. The maximum value of $h$ is
		\begin{equation*}
			h({t}^*)=\frac{\t-2}{2(\t-(p+q)\delta_{p+q})}\left(\frac{\t \mathcal{S}(a,b)^{\t/2}(2-(p+q)\delta_{p+q})}{2(\t-(p+q)\delta_{p+q})}\right)^{\frac{2-(p+q)\delta_{p+q}}{\t-2}}.
		\end{equation*}
		Consequently, $f$ is positive on an interval $({\kappa}_1,{\kappa}_2)$ if and only if $h({t}^*)>\beta C_{a,b}^{p+q}(\rho_1^2+\rho_2^2)^{\frac{(p+q)(1-\delta_{p+q})}{2}}$. This leads to  $\beta<\beta_1$ where 
		\begin{equation}\label{d200}
			\beta_1:=\frac{\t-2}{2 C_{a,b}^{p+q}(\rho_1^2+\rho_2^2)^{\frac{(p+q)(1-\delta_{p+q})}{2}}(\t-(p+q)\delta_{p+q})}\left(\frac{\t \mathcal{S}(a,b)^{\t/2}(2-(p+q)\delta_{p+q})}{2(\t-(p+q)\delta_{p+q})}\right)^{\frac{2-(p+q)\delta_{p+q}}{\t-2}}.
		\end{equation}
		Therefore, the desired result follows. \qed
	\end{proof}
	\begin{lemma}\label{D3}
		For $0<\beta<\beta_1$ and $\u\in S\r$, the function $\Phi_\beta^\u$ admits exactly two critical points $t_\beta^{\u}<s_\beta^{\u}$ and possesses two zeros $c_\beta^{\u}<d_\beta^\u$ satisfying
		\begin{align*}  
			t_\beta^{\u}<c_\beta^{\u}<s_\beta^{\u}<d_\beta^{\u}.
		\end{align*}
		Moreover,
		\begin{itemize}
			\item[i)] $\mathcal{N}_\beta\r$ forms a submanifold of $\y$, and $\mathcal{N}_\beta^0\r=\emptyset$.
			\item[ii)] one has $k\star\u\in\mathcal{N}^+_\beta\r$ if and only if $k=t_\beta^\u$, while $k\star\u\in\mathcal{N}^-_\beta\r$ if and only if $k=s_\beta^\u$.
			\item[iii)] for all $k\le c_\beta^\u$,   the following holds 
			\begin{align*}  
				\left(\||x|^{-a}\na(k\star u)\|_2^2+\||x|^{-a}\na(k\star v)\|_2^2\right)^{1/2}\le \kappa_1, 
			\end{align*}
			\begin{align*}  
				\notag E_\beta(&t_\beta^\u\star \u)\\\notag
				&=\min \big\{E_\beta(k\star\u)\mid k\in \R~ \text{and} ~\left(\||x|^{-a}\na(k\star u)\|_2^2+\||x|^{-a}\na(k\star v)\|_2^2\right)^{1/2}\le \kappa_1 \big\}<0.
			\end{align*}
			\item[iv)] the function $\Phi_\beta^\u(t)$ is strictly decreasing on $(s_\beta^\u,+\infty)$, and
			\begin{align*}  
				\Phi_\beta^\u(s_\beta^\u)=\max_{t\in\R}\Phi_\beta^\u(t)>0.
			\end{align*}
			\item[v)] the maps $\u\mapsto t_\beta^\u$ and $\u\mapsto s_\beta^\u$ are of class $C^1$.
		\end{itemize}
	\end{lemma}
	\begin{proof}
		Let $\u\in S\r$. Then
		\begin{align*}  
			\Phi_\beta^\u(k)=E_\beta(k\star\u)\ge f\left(e^k(\||x|^{-a}\na u\|_2^2+\||x|^{-a}\na v\|_2^2)^{1/2}\right).
		\end{align*}
		As a consequence,
		\begin{align*}  
			\Phi_\beta^\u(k)>0,\quad \forall\,k\in\left(\log\frac{\kappa_1}{(\||x|^{-a}\na u\|_2^2+\||x|^{-a}\na v\|_2^2)^{1/2}},\log\frac{\kappa_2}{(\||x|^{-a}\na u\|_2^2+\||x|^{-a}\na v\|_2^2)^{1/2}}\right),
		\end{align*}
		provided that $0<\beta<\beta_1$. Moreover, since
		\begin{align*}  
			\Phi_\beta^\u(-\infty)=0^-\quad\text{and}\quad \Phi_\beta^\u(+\infty)=-\infty,
		\end{align*}
		the function $\Phi_\beta^\u(s)$ has at least two critical points $t_\beta^\u<s_\beta^\u$. Here $s_\beta^\u$ corresponds to a global maximum at a positive level, whereas $t_\beta^\u$ corresponds to a local minimum at a negative level. By \cite[Remark 3.1]{li2021normalized}, $\Phi_\beta^\u(k)$ has at most two critical points in $\R$. If $k\star\u\in\mathcal{N}_\beta\r$, then $k=t_\beta^\u$ or $k=s_\beta^\u$, because $(\Phi_\beta^\u)'(k)=P_\beta(k\star\u)$. Arguing as in \cite[Lemma 3.2]{li2021normalized}, for $0<\beta<\beta_1$ we obtain that $\mathcal{N}_\beta^0\r=\emptyset$ and $\mathcal{N}_\beta\r$ is a submanifold of $\y$. Consequently, $t_\beta^\u\star\u\in \mathcal{N}_\beta^+\r$ and $s_\beta^\u\star\u\in \mathcal{N}_\beta^-\r$. By the monotonicity properties of $\Phi_\beta^\u$, it follows that the function has exactly two zeros $c_\beta^\u$ and $d_\beta^\u$ satisfying $t_\beta^\u<c_\beta^\u<s_\beta^\u<d_\beta^\u$. Finally, since
		$(\Phi_\beta^{\u})'(t_\beta^{\u})=0$
		and
		$(\Phi_\beta^{\u})''(t_\beta^{\u})\neq 0$
		(and similarly for $s_\beta^{\u}$),
		the implicit function theorem implies that the maps
		\begin{align*}  
			\u\mapsto t_\beta^{\u}
			\quad\text{and}\quad
			\u\mapsto s_\beta^{\u}
		\end{align*}
		are of class $C^1$.
		\qed
	\end{proof}
	For $R>0$, we define
	\begin{align*}  \mathcal{G}_R\r:=\{\u\in S\r|(\||x|^{-a}\na u\|_2^2+\||x|^{-a}\na v\|_2^2)^{1/2}<R\}.\end{align*}
	\begin{lemma}\label{D4}
		Let $0<\beta<\beta_1$. Then the following assertions hold:
		\begin{itemize}
			\item[(i)] 
			$\displaystyle
			m_\beta{\r}
			=
			\inf_{\mathcal{G}_{\kappa_1}{\r}} E_\beta\u
			<0.
			$
			\item[(ii)] 
			For any $0<\tilde{\rho}_1\le \rho_1$ and $0<\tilde{\rho}_2\le \rho_2$, one has
			\begin{align*}  
				m_\beta\r
				\le
				m_\beta{\rt}
				.
			\end{align*}
		\end{itemize}
	\end{lemma}
	\begin{proof}
		$(i)$ By Lemma~\ref{D3}, we have $\mathcal{N}_\beta^+\r\subset \mathcal{G}_{\kappa_1}\r$. Moreover, 
		\begin{align*}  m_\beta\r=\inf_{\mathcal{N}_{\beta}\r}E_\beta\u=\inf_{\mathcal{N}_{\beta}^+\r}E_\beta\u<0.\end{align*}
		Since $\mathcal{N}_\beta^{+}\r \subset \mathcal{G}_{\kappa_1}{\r}$, it follows that
		$\displaystyle m_\beta\r\ge\inf_{\mathcal{G}_{\kappa_1}\r}E_\beta\u$.
		On the other hand, for any $\u \in \mathcal{G}_{\kappa_1}{\r}$, 
		Lemma~\ref{D3} ensures that there exists $t_\beta^{\u}\in\R$ such that
		$t_\beta^{\u}\star \u \in \mathcal{N}_\beta^{+}\r$ and
		\begin{align*}  m_\beta\r\le E_\beta(t_\beta^\u\star\u)\le E_\beta\u.\end{align*} Therefore, $m_\beta\r=\inf_{\mathcal{G}_{\kappa_1}\r}E_\beta\u$.\\
		\textit{(ii)} Recalling \eqref{d9}, we obtain
		\begin{equation*}
			m_\beta\r=\inf_{\mathcal{G}_{t^*}\r} E_\beta.
		\end{equation*}
		Let $\var>0$ be arbitrary. Choose $\u\in \mathcal{G}_{t^*}\rt $ such that $E_\beta\u\le m_\beta\rt+\frac{\var}{2}$ and $\phi\in C_0^\infty(\R^N)$ be a cutoff function satisfying
		\begin{align*}  0\le\phi\le1\quad\text{and}\quad\phi(x)=\begin{cases}
				0,\quad \text{if}~|x|\ge 2,\\
				1,\quad\text{if}~|x|\le 1.
		\end{cases}\end{align*} 
		For $\delta>0$, define $(u_\delta(x),v_\delta(x))=(u(x)\phi(\delta x),v(x)\phi(\delta x))$. Then $(u_\delta,v_\delta)\to\u$ in $\y$ as $\delta\to0.$ Consequently, for small $\eta>0,$ there exists $\delta>0$ sufficiently small such that
		\begin{equation}\label{d11}
			E_\beta(u_\delta,v_\delta)\le E_\beta(u,v)+\frac{\var}{4}
			\quad\text{and}\quad \left(\||x|^{-a}\na u_\delta\|_2^2+\||x|^{-a}\na v_\delta\|_2^2\right)^{1/2}<t^*-\eta.
		\end{equation}
		Let $\psi\in C_0^\infty(\R^N)$ such that $\text{supp}(\psi)\subset B(0,1+\frac{4}{\delta})\backslash B(0,\frac{4}{\delta})$, where $B(0,R)$ denotes the ball of radius $R$ centered at the origin. Define
		\begin{align*}  w_{\rho_1}=\frac{\sqrt{\rho_1^2-\||x|^{-a}u_\delta\|_2^2}}{\||x|^{-a}\psi\|_2\phantom{^2}}\psi\quad\text{and}\quad w_{\rho_2}=\frac{\sqrt{\rho_2^2-\||x|^{-a}v_\delta\|_2^2}}{\||x|^{-a}\psi\|_2\phantom{^2}}\psi.\end{align*}
		By construction, the supports are disjoint
		\begin{align*}  (\text{supp}(u_\delta)\cup\text{supp}(v_\delta))\cap(\text{supp}(k\star w_{\rho_1})\cup\text{supp}(k\star w_{\rho_2}))=\emptyset,\end{align*}
		for $k<0.$ Hence, $(u_\delta+k\star w_{\rho_1},v_\delta+k\star w_{\rho_2})\in S\r$. As $k\to-\infty$, we obtain
		\begin{equation}\label{d10}
			E_\beta(k\star (w_{\rho_1},w_{\rho_2}))\le \frac{\var}{4}\quad\text{and}\quad \left(\||x|^{-a}\na (k\star w_{\rho_1})\|_2^2+\||x|^{-a}\na (k\star w_{\rho_2})\|_2^2\right)^{1/2}\le \frac{\eta}{2}.
		\end{equation}
		It follows that
		\begin{align*}  \left(\||x|^{-a}\na( u_\delta+k\star w_{\rho_1})\|_2^2+\||x|^{-a}\na (v_\delta+k\star w_{\rho_2})\|_2^2\right)^{1/2}<t^*.\end{align*}
		Combining \eqref{d11} and \eqref{d10}, we deduce
		\begin{align*}  
			m_\beta\r&\le E_\beta(u_\delta+k\star w_{\rho_1}, v_\delta+k\star w_{\rho_2})\\
			&=E_\beta(u_\delta,v_\delta)+E_\beta(k\star w_{\rho_1},k\star w_{\rho_2})\\
			&\le m_\beta\rt+\var.
		\end{align*}
		This completes the proof.\qed
	\end{proof}
	\noindent
	\textbf{Proof of Theorem~\ref{DD2}.}
	Let $(\hat{u}_n,\hat{v}_n)\in \mathcal{G}_{\kappa_1}{\r}$ be a minimizing sequence for $m_\beta{\r}$. 
	By replacing $(\hat{u}_n,\hat{v}_n)$ with $(|\hat{u}_n|,|\hat{v}_n|)$, we may assume that $(\hat{u}_n,\hat{v}_n)$ is nonnegative.
	
	Moreover, since
	\begin{align*}  
		E_\beta(t_\beta^{(\hat{u}_n,\hat{v}_n)}\star (\hat{u}_n,\hat{v}_n))
		\le
		E_\beta(\hat{u}_n,\hat{v}_n),
	\end{align*}
	we may further assume that 
	\begin{align*}  
		(\hat{u}_n,\hat{v}_n)\in \mathcal{N}_\beta^{+}\r.
	\end{align*}
	
	By Ekeland's variational principle, there exists a Palais–Smale sequence $\v$ for $E_\beta\big|_{S\r}$ such that
	\begin{align*}  
		\|\v-(\hat{u}_n,\hat{v}_n)\|_{\y}\to0
		\quad \text{as } n\to\infty.
	\end{align*}
	Consequently,
	\begin{align*}  
		P_\beta\v
		=
		P_\beta(\hat{u}_n,\hat{v}_n)
		+
		o(1)
		\to 0,
		\quad \text{and} \quad
		u_n^-,v_n^- \to 0
		\ \text{a.e. in } \R^N.
	\end{align*}
	
	Applying Proposition~\ref{D5} with $c=m_\beta{\r}$ (its monotonicity hypothesis holds by Lemma~\ref{D4}(ii), and \eqref{d2} holds because $m_\beta\r<0$, so that $\min\{0,m_\beta\r\}=m_\beta\r$), we deduce that
	\begin{align*}  
		\v\to\u \quad \text{in } \y,
	\end{align*}
	and
	\begin{align*}  
		(\lambda_{1,n},\lambda_{2,n})
		\to
		(\lambda_1,\lambda_2)
		\in \R^+\times\R^+.
	\end{align*}
	By strong convergence, $\u\in \mathcal{N}_\beta{\r}$ is a solution of \eqref{d1}, and hence a normalized ground state.\qed
	\subsection{Mass-supercritical case}
	In this subsection, we turn our attention to the mass-supercritical regime, corresponding to 
	$(p + q) > (p + q)_c = \frac{2N + 4}{N - 2a + 2b}$, which is equivalent to $(p + q)\delta_{p+q} > 2$. 
	In contrast to the mass-subcritical case, the energy functional $E_\beta$ is unbounded from below 
	on  $S\r$, and the ground-state solution cannot be obtained as a local minimizer. 
	Therefore, we employ a minimax characterization over the Pohozaev manifold. To understand the 
	underlying geometry and construct a suitable Palais--Smale sequence, we first analyze the behavior 
	of the fibering map $\Phi_\beta^{(u,v)}$ and the structure of the Pohozaev manifold $\mathcal{N}_\beta(\rho_1, \rho_2)$.
	\begin{lemma}\label{D6}
		For $\beta>0$, the fibering map $\Phi_\beta^{\u}$ admits exactly one critical point 
		$t_\beta^{\u}$ for every $\u \in S{\r}$. 
		In addition, the following properties hold:
		\begin{itemize}
			\item[(i)] $\mathcal{N}_\beta\r=\mathcal{N}^-_\beta\r$ and $\mathcal{N}_\beta\r $ is a submanifold of $\y$.
			\item[(ii)] $t\star\u\in \mathcal{N}_\beta\r$ if and only if $t=t_\beta^\u.$
			\item[(iii)] $\Phi_\beta^\u(t)$ is strictly decreasing and concave on $(t_\beta^\u,+\infty)$ and 
			\begin{align*}  \Phi_\beta^\u(t_\beta^\u)=\max_{t\in\R}\Phi_\beta^\u(t)>0.\end{align*}
			\item[(iv)] The map $\u\mapsto t_\beta^\u$ is of class $C^1$.
		\end{itemize}
	\end{lemma}
	\begin{proof}
		Let $\u\in\mathcal{N}_\beta\r\backslash\mathcal{N}_\beta^-\r$. Then
		\begin{equation}\label{d13}
			\||x|^{-a}\na u\|_2^2+\||x|^{-a}\na v\|_2^2=\||x|^{-b}u\|_\t^\t+\||x|^{-b}v\|_\t^\t+\beta \pq\b\u,\end{equation}
		and \begin{equation}\label{d12}
			2\left(\||x|^{-a}\na u\|_2^2+\||x|^{-a}\na v\|_2^2\right)\ge\t\left(\||x|^{-b}u\|_\t^\t+\||x|^{-b}v\|_\t^\t\right)+\beta (p+q)^2\delta_{p+q}^2\b\u.\end{equation}
		Combining \eqref{d13} and \eqref{d12}, we have
		\begin{align*}  0\ge (\t-2)\left(\||x|^{-b}u\|_\t^\t+\||x|^{-b}v\|_\t^\t\right)+\pq(\pq-2)\b\u.\end{align*}
		This yields a contradiction, since $\pq>2$. Hence, $\mathcal{N}_\beta\r=\mathcal{N}_\beta^-\r$. Moreover, arguing as in \cite[Lemma 3.2]{li2021normalized}, one concludes that  $\mathcal{N}_\beta\r$ is a submanifold of $\y$. The proofs of assertions  \textit{(ii)--(iv)} follow by repeating the arguments used in Lemma~\ref{D3}, and are therefore omitted.\qed
	\end{proof}
	\begin{lemma}\label{D7}
		$\displaystyle m_\beta\r=\inf_{S\r}\max_{k\in\R}E_\beta(k\star\u)$.
	\end{lemma}
	\begin{proof}
		By Lemma~\ref{D6}, for every $\u\in S\r$, the map $k\mapsto E_\beta(k\star\u)$
		admits a unique maximizer and
		\begin{align*}  \max_{k\in\R}E_\beta(k\star\u)=E_\beta(t_\beta^\u\star\u)\ge m_\beta\r.\end{align*}
		Taking the infimum over $S\r$, we obtain
		\begin{align*}  m_\beta\r\le \inf_{S\r}\max_{k\in\R}E_\beta(k\star\u).\end{align*}
		Conversely, if 
		$\u\in \mathcal{N}_\beta\r$, then again by Lemma~\ref{D6},
		\begin{align*}  E_\beta\u=\max_{k\in\R}E_\beta(k\star\u)\ge\inf_{S\r}\max_{k\in\R}E_\beta(k\star\u).\end{align*}
		Combining the two inequalities yields the desired identity.\qed
	\end{proof}
	\begin{lemma}\label{D9}
		There exists a Palais-Smale sequence $\v\subset S\r\cap\y$ for $E_\beta\big|_{S\r}$ at the level $m_\beta\r$, with $P_\beta\v\to0$ and $u_n^-,~v_n^-\to0$ a.e. in $\R^N$ as $n\to\infty.$
	\end{lemma}
	\begin{proof}
		Define $\overline{E}_\beta(k,\u):=E_\beta(k\star\u)$. By Lemma~\ref{D7}, the level $m_\beta\r$ admits the minimax characterization
		\begin{align*}  m_\beta\r=\inf_{\gamma\in\Gamma_\beta}\max_{t\in[0,1]}\overline{E}_\beta(\gamma(t)),\end{align*}
		where
		\begin{align*}  
			\Gamma_\beta:=&\big\{\gamma:[0,1]\to \R\times(S\r\cap\y)|\gamma~\text{is continuous,}~\gamma(0)\in\{0\}\times \mathcal{G}_r\r,\\&\quad\gamma(1)\in \{(0,\u)|~E_\beta\u\le0\}\big\}
		\end{align*}
		for $r>0$ arbitrarily small. The classical minimax principle  \cite{bartsch2016normalized} therefore yields a Palais–Smale sequence at the level $m_\beta\r.$
		
		Applying to such a minimizing sequence of paths the minimax principle for homotopy-stable families \cite[Theorem~3.2]{ghoussoub1993duality}, we obtain a Palais--Smale sequence $\{(k_n,(w_n,z_n))\}$ for $\overline E_\beta$ at the level $m_\beta\r$ with $k_n\to0$ and $w_n^-,z_n^-\to0$ in $X$. Setting $\v:=k_n\star(w_n,z_n)$ and using $\partial_k\overline E_\beta(k,\u)=P_\beta(k\star\u)$ together with $k_n\to0$, which makes the map $(\varphi,\psi)\mapsto k_n\star(\varphi,\psi)$ uniformly bounded on $\y$, we get a Palais--Smale sequence for $E_\beta|_{S\r}$ with the stated properties. The details are as in \cite[Lemma~3.6]{bartsch2016normalized} and \cite[Lemma~2.6]{bartsch2023existence}, and carry over because the dilation $k\star$ acts here exactly as in the unweighted case.\qed
	\end{proof}
	\begin{lemma}\label{D11}
		For given $\rho_1,~\rho_2>0$, the following hold
		\begin{itemize}
			\item[(i)] $m_\beta\r\le m_\beta\rt$ for $0<\tilde{\rho}_1<\rho_1$ and $0<\tilde{\rho}_2<\rho_2.$
			\item[(ii)] $m_\beta\r$ is nonincreasing with respect to $\beta\in(0,+\infty).$
			\item[(iii)] $\displaystyle\lim_{\beta\to+\infty}m_\beta\r=0^+.$
		\end{itemize}
	\end{lemma}
	\begin{proof}\textit{(i)} It suffices to prove that for every $\var>0$,
		\begin{equation}\label{d17}
			m_\beta\r<m_\beta\rt+\var.
		\end{equation}
		Using the definition of $m_\beta\rt$, there exists $\u\in\mathcal{N}_\beta\rt$ such that
		\begin{equation}\label{d18}
			E_\beta\u\le m_\beta\rt+\frac{\var}{2}.
		\end{equation}
		Let $\phi\in C_0^\infty(\R^N)$ be a radial function satisfying 
		\begin{align*}  0\le\phi\le1\quad\text{and}\quad\phi(x)=\begin{cases}
				0,\quad \text{if}~|x|\ge 2,\\
				1,\quad\text{if}~|x|\le 1.
		\end{cases}\end{align*} 
		For $\delta>0$, define $(u_\delta(x),v_\delta(x))=(u(x)\phi(\delta x),v(x)\phi(\delta x))$. Then $(u_\delta,v_\delta)\to\u$ in $\y$ as $\delta\to0^+.$ By Lemma~\ref{D6}~\textit{(iv)}, $t_\beta^{(u_\delta,v_\delta)}\star(u_\delta,v_\delta)\to t_\beta^\u\star\u$ in $\y$ as $\delta\to 0^+$. Hence,  there is $\delta>0$ small such that
		\begin{equation}\label{d19}           E_\beta(t_\beta^{(u_\delta,v_\delta)}\star(u_\delta,v_\delta))\le E_\beta\u+\frac{\var}{4}.
		\end{equation}
		Next, choose $\psi\in C_0^\infty(\R^N)$ such that $\text{supp}(\psi)\subset \R^N\backslash B(0,\frac{4}{\delta})$. Define
		\begin{align*}  w_{\rho_1}=\frac{\sqrt{\rho_1^2-\||x|^{-a}u_\delta\|_2^2}}{\||x|^{-a}\psi\|_2\phantom{^2}}\psi\quad\text{and}\quad w_{\rho_2}=\frac{\sqrt{\rho_2^2-\||x|^{-a}v_\delta\|_2^2}}{\||x|^{-a}\psi\|_2\phantom{^2}}\psi.\end{align*}
		By construction, the supports are disjoint
		\begin{align*}  (\text{supp}(u_\delta)\cup\text{supp}(v_\delta))\cap(\text{supp}(k\star w_{\rho_1})\cup\text{supp}(k\star w_{\rho_2}))=\emptyset\quad\text{for}~k\le 0,\end{align*}
		and thus\begin{align*}  (\tilde{u}_k,\tilde{v}_k):=(u_\delta+k\star w_{\rho_1},v_\delta+k\star w_{\rho_2})\in S\r.\end{align*}
		Let $t_k:=t_\beta^{(\tilde{u}_k,\tilde{v}_k)}$. Then, $P_\beta(t_k\star(\tilde{u}_k,\tilde{v}_k))=0$, i.e.
		\begin{align*}  \frac{1}{e^{(\t-2)t_k}}\left(\||x|^{-a}\na \tilde{u}_k\|_2^2+\||x|^{-a}\na\tilde{v}_k\|_2^2\right)=\frac{\beta\pq}{e^{(\t-\pq)t_k}}\b(\tilde{u}_k,\tilde{v}_k)\\+\||x|^{-b}\tilde{u}_k\|_\t^\t+\||x|^{-b}\tilde{v}_k\|_\t^\t.\end{align*}
		Since $(\tilde{u}_k,\tilde{v}_k)\to(u_\delta,v_\delta)\ne(0,0)$ as $k\to-\infty$, we deduce $\displaystyle\limsup_{k\to-\infty}~t_k<+\infty.$
		Consequently, $t_k+k\to-\infty$ as $k\to-\infty$, and therefore  for $k<-1$ with $|k|$ large, 
		\begin{equation}\label{d20}
			E_\beta((t_k+k)\star(w_{\rho_1},w_{\rho_2}))<\frac{\var}{4}.
		\end{equation}
		Combining \eqref{d18}, \eqref{d19}, and \eqref{d20}, we obtain
		\begin{align*}  
			m_\beta\r&\le E_\beta(t_k\star (\tilde{u}_k,\tilde{v}_k))=E_\beta(t_k\star(u_\delta,v_\delta))+E_\beta((t_k+k)\star(w_{\rho_1},w_{\rho_2}))\\
			&\le E_\beta(t_\beta^{(u_\delta,v_\delta)}\star(u_\delta,v_\delta))+\frac{\var}{4}\le m_\beta\rt+\var,
		\end{align*}
		which proves \eqref{d17}.
		
		\noi
		\textit{(ii)} Let $0<\beta'<\beta$. The functional $E_\beta$ is  nonincreasing with respect to $\beta$, since $\b\ge0$. Hence,
		\begin{align*}  m_\beta\r=\inf_{S\r}\max_{t\in\R}E_\beta(t\star\u)\le\inf_{S\r}\max_{t\in\R}E_{\beta'}(t\star\u)=m_{\beta'}\r.\end{align*}
		Therefore, $m_\beta\r$ is non-increasing with respect to $\beta\in(0,+\infty)$.
		
		\noi
		\textit{(iii)} We first show that $m_\beta\r>0$ for every $\beta>0$. Let $\u\in\mathcal{N}_\beta\r$. Then, using the definition of $\mathcal{S}(a,b)$ and inequality \eqref{d5}, there exist $C_1,~C_2>0,$ such that
		\begin{align*}  
			&\||x|^{-a}\na u\|_2^2+\||x|^{-a}\na v\|_2^2\\ &=\beta\pq\b\u+\||x|^{-b}u\|_\t^\t+\||x|^{-b}v\|_\t^\t\\
			&\le C_1\left(\||x|^{-a}\na u\|_2^2+\||x|^{-a}\na v\|_2^2\right)^{\frac{\t}{2}}+C_2\left(\||x|^{-a}\na u\|_2^2+\||x|^{-a}\na v\|_2^2\right)^{\frac{\pq}{2}},
		\end{align*}
		it follows that $\displaystyle\inf_{\mathcal{N}_\beta\r}\left(\||x|^{-a}\na u\|_2^2+\||x|^{-a}\na v\|_2^2\right)>0$. Using the identity $P_\beta\u=0$ on $\mathcal{N}_\beta\r$, we compute
		\begin{align*}      m_\beta\r&=\inf_{\mathcal{N}_\beta\r}E_\beta\u=\inf_{\mathcal{N}_\beta\r}E_\beta\u-\frac{1}{\pq}P_\beta\u\\
			&=\inf_{\mathcal{N}_\beta\r}\frac{\pq-2}{2\pq}\left(\||x|^{-a}\na u\|_2^2+\||x|^{-a}\na v\|_2^2\right)\\
            & \qquad +\frac{\t-\pq}{\pq\t}\left(\||x|^{-b}u\|_\t^\t+\||x|^{-b}v\|_\t^\t\right)\\
			&\ge C\inf_{\mathcal N_\beta\r}\left(\||x|^{-a}\na u\|_2^2+\||x|^{-a}\na v\|_2^2\right)>0.
		\end{align*}
		We now prove \textit{(iii)}. It suffices to show that for every  $\var>0$
		there exists $\bar{\beta}>0$ such that
		\begin{equation}\label{d21}
			m_\beta\r<\var\quad \text{for any}~\beta\ge\bar{\beta}.
		\end{equation}
		Choose $\phi\in C_0^\infty(\R^N)$ such that $\||x|^{-a}\phi\|_2\le\min\{\rho_1,~\rho_2\}$. Then
		\begin{align}\label{d22}
			\notag m_\beta\r&\le m_\beta(\||x|^{-a}\phi\|_2,\||x|^{-a}\phi\|_2)\le \max_{t\in\R}E_\beta(t\star\phi,t\star\phi)\\
			&=\max_{t\in \R}\left(2I(t\star\phi)-\beta e^{\pq t}\||x|^{-b}\phi\|_{p+q}^{p+q}\right),
		\end{align}
		where $I(u):=\frac{1}{2}\||x|^{-a}\na u\|_2^2-\frac{1}{\t}\||x|^{-b}u\|_\t^\t.$
		Since $I(t\star \phi)\to0^+$ as $t\to-\infty,$  there exists $t_0$ such that $I(t\star \phi)<\frac{\var}{4}$ for any $t<-t_0$. On the other hand, choosing $\bar{\beta}>0$ sufficiently large,
		\begin{align*}  
			\max_{t\ge-t_0}\left(2I(t\star\phi)-\beta e^{\pq t}\||x|^{-b}\phi\|_{p+q}^{p+q}\right)&\le \frac{2d\||x|^{-a}\na\phi\|_2^{N/d}}{N\||x|^{-b}\phi\|_\t^{N/d}}-\beta e^{-\pq t_0}\||x|^{-b}\phi\|_{p+q}^{p+q}\\
			&<\var\quad \text{for}~\beta\ge\bar{\beta}.
		\end{align*}
		Thus, $\max_{t\in\R}\left(2I(t\star\phi)-\beta e^{\pq t}\||x|^{-b}\phi\|_{p+q}^{p+q}\right)<\var$ when $\beta\ge\bar{\beta}$, which together with \eqref{d22} yields \eqref{d21}.
		
		The proof is complete.
		\qed
	\end{proof}
	From Lemma~\ref{D11} \textit{(iii)}, it follows that
	\begin{equation*}
		\beta_0:=\inf\big\{\beta>0|~m_\beta\r<\frac{d}{N}\mathcal{S}(a,b)^{N/2d}\big\}<+\infty.
	\end{equation*}
	To determine whether $\beta_0=0$ or $\beta_0>0$, 
	we employ the test functions $\zeta_\varepsilon=\phi U_\varepsilon$, 
	where $U_\varepsilon$ is defined in \eqref{c96} and 
	$\phi\in C_0^\infty(\mathbb{R}^N)$ is a radial cut-off function satisfying
	\begin{align*}  0\le\phi\le1\quad\text{and}\quad\phi(x)=\begin{cases}
			0,\quad \text{if}~|x|\ge 2,\\
			1,\quad\text{if}~|x|\le 1.
	\end{cases}\end{align*} 
	\begin{lemma}{\cite[Proposition A.1]{goel2026normalized}}\label{D19}
		Let $\zeta_\varepsilon=\phi U_\varepsilon$. Then the following estimates hold as $\varepsilon\to0^+$
		\begin{itemize}
			\item[(i)] $\||x|^{-a}\na \zeta_{\varepsilon}\|_2^2= \sr^{\frac{N}{2d}}+O(\varepsilon^{\frac{N-2d}{2d}}),
			$ 
			\item[(ii)] $\||x|^{-b}\zeta_{\varepsilon}\|_{\t}^{\t}=\sr^{\frac{N}{2d}}+O(\varepsilon^{\frac{N}{2d}}),$
			\item[(iii)] \begin{equation*}
				\||x|^{-b}\zeta_{\varepsilon}\|_q^q=  \begin{cases}
					O\left(\varepsilon^{\frac{N-2d}{4d}}\right),&\text{if}~~2<q<\frac{N}{N-2(1+a)+b},\\
					O\left(\varepsilon^{\frac{(2N-Nq+2qd)(N-2d)}{4d(N-2-2a)}}|\log\var|\right),&\text{if}~~q=\frac{N}{N-2(1+a)+b},\\
					O\left(\varepsilon^{\frac{(2N-Nq+2qd)(N-2d)}{4d(N-2-2a)}}\right),&\text{if}~\frac{N}{N-2(1+a)+b}<q<\t.
				\end{cases}
			\end{equation*}
		\end{itemize}
	\end{lemma}
	\begin{lemma}\label{D201}
		Suppose either $\max\{0,\frac{N-4}{2}\}< a<\frac{N-2}{2}$ or $a=\frac{N-4}{2}$ (with $N>4$), and $
		r:=\min\{p,q\}.
		$ Then, the following assertions hold
		\begin{itemize}
			\item[(I)] $m_\beta\r\le\frac{d}{N}\mathcal{S}(a,b)^{N/2d}$, for all $\beta>0$.
			\item[(II)] $\beta_0=0$ if one of the following conditions  holds: \begin{itemize}\item[(i)]  $a=\frac{N-4}{2}$ with $N>4$ and $r<2$.
				\item[(ii)]  $\max\{0,\frac{N-4}{2}\}<a<\frac{N-2}{2}$ ,$r<2$ and
				\begin{align*}  \max\Big\{\p,\frac{N}{N-2-2a+b},\t-\frac{2(2-r)(N-2a-2)}{4+2a-N}\Big\}<p+q<\t,\end{align*}
				or
				\begin{align*}  \p<p+q<\min\Big\{\frac{N}{N-2-2a+b},\t-\frac{2(r-1)}{N-2(1+a-b)}\Big\},\end{align*}
				or\begin{align*}  \p<p+q=\frac{N}{N-2-2a+b}.\end{align*} 
			\end{itemize}
			\item[(III)] for $p,~q\ge2$, $\beta_0>0$.
			\item[(IV)] if $0<\beta<\beta_0$, then $m_\beta\r=\frac{d}{N}\mathcal{S}(a,b)^{N/2d}$ is not attained. 
		\end{itemize}
		
	\end{lemma}
	
	\begin{proof} \textbf{(I)} We may assume without loss of generality that $p>q$. 
		Let $\theta>0$ be a parameter to be specified later. 
		For $\varepsilon>0$, define
		
		\begin{align*}  u_\var=\frac{\rho_1}{\||x|^{-a}\zeta_\var\|_2}\zeta_\var\quad\text{and}\quad v_\var=\frac{\var^\theta}{\||x|^{-a}\zeta_\var\|_2}\zeta_\var.\end{align*}
		Then
		\begin{align*}  
			\||x|^{-a}u_\varepsilon\|_2=\rho_1,
			\qquad
			\||x|^{-a}v_\varepsilon\|_2=\varepsilon^{\theta},
		\end{align*}
		so that $(u_\varepsilon,v_\varepsilon)\in S(\rho_1,\varepsilon^{\theta})$. Let $t_\var:=t_\beta^{(u_\var,v_\var)}$ be the parameter given by Lemma~\ref{D6}. Then $t_\var\star (u_\var,v_\var)\in\mathcal{N}_\beta(\rho_1,\var^\theta)$. Choose $\var>0$ sufficiently small so that $\var^\theta<\rho_2$. Then by Lemma~\ref{D11}\textit{(i)} 
        \begin{align}\label{d29}
		\notag m_\beta\r&\le m_\beta(\rho_1,\var^\theta)\le E_\beta(t_\var\star(u_\var,v_\var))\\
			\notag &=\frac{e^{2t_\var}}{2}\left(\||x|^{-a}\na u_\var\|_2^2+\||x|^{-a}\na v_\var\|_2^2\right)-\frac{e^{\t t_\var}}{\t}\left(\||x|^{-b}u_\var\|_\t^\t+\||x|^{-b}v_\var\|_\t^\t\right)\\
          \notag  & \qquad -\beta e^{\pq t_\var}\b(u_\var,v_\var)\\
			\notag&\le\max_{k>0}\left(\frac{(\rho_1^2+\var^{2\theta})}{2}k^2\frac{\||x|^{-a}\na\zeta_\var\|_2^2}{\||x|^{-a}\zeta_\var\|_2^2}-\frac{\rho_1^\t+\var^{\t\theta}}{\t}k^\t\frac{\||x|^{-b}\zeta_\var\|_\t^\t}{\||x|^{-a}\zeta_\var\|_2^\t}\right)\\
           \notag & \qquad -C\var^{q\theta}e^{\pq t_\var}\frac{\||x|^{-b}\zeta_\var\|_{p+q}^{p+q}}{\||x|^{-a}\zeta_\var\|_2^{p+q}}\\
			\notag&= \frac{d}{N}\left(\frac{(\rho_1^2+\var^{2\theta})\||x|^{-a}\na \zeta_\var\|_2^2}{(\rho_1^\t+\var^{\t\theta})^{2/\t}\||x|^{-b}\zeta_\var\|_\t^2}\right)^{N/2d}-C\var^{q\theta}e^{\pq t_\var}\frac{\||x|^{-b}\zeta_\var\|_{p+q}^{p+q}}{\||x|^{-a}\zeta_\var\|_2^{p+q}}\\
			&=\frac{d}{N}\mathcal{S}(a,b)^{N/2d}+O(\var^{\frac{N-2d}{2d}})+O(\var^{2\theta})-C\var^{q\theta}e^{\pq t_\var}\frac{\||x|^{-b}\zeta_\var\|_{p+q}^{p+q}}{\||x|^{-a}\zeta_\var\|_2^{p+q}}.
		\end{align}
		Next, we claim that there exists $C>0$ such that
		$e^{t_\var}\ge C\||x|^{-a}\zeta_\var\|_2$ as $\var\to0$. Since $P_\beta(t_\var\star (u_\var,v_\var))=0$, we have
		\begin{align*}  e^{(\t-2)t_\var}\left(\||x|^{-b}u_\var\|_\t^\t+\||x|^{-b}v_{\var}\|_\t^\t\right)\le\||x|^{-a}\na u_\var\|_2^2+\||x|^{-a}\na v_\var\|_2^2.\end{align*}
		It gives \begin{equation}\label{d24}
			e^{t_\var}\le\left(\frac{\||x|^{-a}\na u_\var\|_2^2+\||x|^{-a}\na v_\var\|_2^2}{\||x|^{-b}u_\var\|_\t^\t+\||x|^{-b}v_{\var}\|_\t^\t}\right)^{\frac{1}{\t-2}}.
		\end{equation}
		Using the fact that $\pq>2$ and \eqref{d24}, for small $\var>0$, we obtain
		\begin{align}\label{d25}
			\notag e^{(\t-2)t_\var}&=\frac{\||x|^{-a}\na u_\var\|_2^2+\||x|^{-a}\na v_\var\|_2^2}{\||x|^{-b}u_\var\|_\t^\t+\||x|^{-b}v_{\var}\|_\t^\t}-\pq\frac{\beta\b(u_\var,v_\var)}{\||x|^{-b}u_\var\|_\t^\t+\||x|^{-b}v_{\var}\|_\t^\t}e^{(\pq-2)t_\var}\\
			\notag&\ge\frac{\||x|^{-a}\na u_\var\|_2^2+\||x|^{-a}\na v_\var\|_2^2}{\||x|^{-b}u_\var\|_\t^\t+\||x|^{-b}v_{\var}\|_\t^\t}\\\notag&\qquad-\pq\frac{\beta\b(u_\var,v_\var)}{\||x|^{-b}u_\var\|_\t^\t+\||x|^{-b}v_{\var}\|_\t^\t}\left(\frac{\||x|^{-a}\na u_\var\|_2^2+\||x|^{-a}\na v_\var\|_2^2}{\||x|^{-b}u_\var\|_\t^\t+\||x|^{-b}v_{\var}\|_\t^\t}\right)^{\frac{\pq-2}{\t-2}}\\
			&\notag=\frac{\||x|^{-a}\na\zeta_\var\|_2^2\||x|^{-a}\zeta_\var\|_2^{\t-2}}{\||x|^{-b}\zeta_\var\|_\t^\t}\Bigg(C_1\\&\qquad-C_2\var^{q\theta}\frac{\||x|^{-b}\zeta_\var\|_{p+q}^{p+q}}{\||x|^{-b}\zeta_\var\|_\t^{\frac{\t(\pq-2)}{\t-2}}\||x|^{-a}\na\zeta_\var\|_2^{\frac{2(\t-\pq)}{\t-2}}\||x|^{-a}\zeta_\var\|_2^{(p+q)(1-\delta_{p+q})}}\Bigg)
		\end{align}
		Using Lemma~\ref{D19}, we have $e^{t_\var}\ge C\||x|^{-a}\zeta_\var\|_2$, substituting this \eqref{d29}, we have our required result.
		
		\noindent \textbf{(II)} We argue as in \cite[Proposition A.2]{goel2026normalized} to estimate
		
		\begin{align*}  
			\frac{
				\var^{q\theta}\||x|^{-b}\zeta_\varepsilon\|_{p+q}^{p+q}
			}{
				\||x|^{-a}\zeta_\varepsilon\|_2^{(p+q)(1-\delta_{p+q})}
			},
		\end{align*}
		
		\noi \textbf{Case $1$.} $(p+q)>\frac{N}{N-2-2a+b}$.
		\begin{equation}\label{d26}
			\frac{\||x|^{-b}\zeta_\var\|_{p+q}^{p+q}}{\||x|^{-a}\zeta_\var\|_{2}^{(1-\delta_{p+q})(p+q)}}=
			\begin{cases}
				C_3|\log\var|^{\frac{(\delta_{p+q}-1)(p+q)}{2}},&\text{if}~0<a=\frac{N-4}{2},\\
				C_3\var^{\frac{(N-2d)(2N-(N-2d)(p+q))(4+2a-N)}{8d(N-2-2a)}},&\text{if}~\max\{\frac{N-4}{2},0\}<a<{\frac{N-2}{2}
				}.
			\end{cases}
		\end{equation}
		Let $\theta\in (\theta_1, \theta_2)$, 
	where $\theta_1= \frac{(N-2d)(2N-(N-2d)(p+q))(4+2a-N)}{8{d}(N-2-2a)(2-q)} $	 and \\$\theta_2= \frac{(N-2d)}{2dq}\left(1-\frac{(2N-(N-2d)(p+q))(4+2a-N)}{4(N-2-2a)}\right)$. Then for $\max\{\frac{N-4}{2},0\}<a<\frac{N-2}{2}$ and $1<q<2$, with $p,~ q$ additionally satisfying
		
		\begin{align*}  
			\t-\frac{2(2-q)(N-2a-2)}{4+2a-N}<p+q,
		\end{align*}
		and, when $0<a=\frac{N-4}{2}$, taking $\theta\in\left(0,\frac{N-2d}{2dq}\right)$.
		
		Using \eqref{d26}, as $\var\to 0$ in \eqref{d25}, we obtain $e^{t_\var}\ge C\||x|^{-a}\zeta_\var\|_2$. Consequently, \eqref{d29} simplifies to
		\begin{align*}  
			m_\beta\r<\frac{d}{N}\mathcal{S}(a,b)^{N/2d}.
		\end{align*}
		Hence, in this setting, the desired result follows.
		
		\noi
		\textbf{Case $2$.} $(p+q)\le\frac{N}{N-2-2a+b}$. 
		
		This case occurs when $0<\max\{\frac{N-4}{2},0\}<a<\frac{N-2}{2}$. In this regime, we have
		\begin{equation}\label{d27}
			\frac{\||x|^{-b}\zeta_\var\|_{p+q}^{p+q}}{\||x|^{-a}\zeta_\var\|_{{2}}^{(1-\delta_{p+q})(p+q)}}=
			\begin{cases}
				C_3\var^{\frac{(N-2d)((p+q)(N-2d)+2(1-N))}{8d}}|\log\var|,&\text{if}~p+q=\frac{N}{N-2(1+a){+}b},\\
				C_3\var^{\frac{(N-2d)((p+q)(N-2d)+2(1-N))}{8d}},&\text{if}~p+q<\frac{N}{N-2(1+a){+}b}.
			\end{cases}
		\end{equation}
		We choose \begin{align*}  \theta\in \left(\frac{(N-2d)((p+q)(N-2d)+2(1-N))}{8d(2-q)},\frac{(N-2d)}{2dq}\left(1-\frac{(N-2d)(p+q)+2(1-N)}{4}\right)\right),\end{align*}
		which is possible since $q<2$ and $p,q$  satisfy 
		\begin{align*}  (p+q)<\t-\frac{2(q-1)}{N-2d}.\end{align*}
		Using  \eqref{d27}, for $\var\to 0$ in \eqref{d25}, we have $e^{t_\var}\ge C\||x|^{-a}\zeta_\var\|_2$. Combining this with \eqref{d29}, we obtain
		\begin{align}
			\notag m_\beta\r< \frac{d}{N}\mathcal{S}(a,b)^{N/2d}.
		\end{align}
		Thus, when the pair $(p,q)$ satisfies the conditions in  either of the above cases, we have $\beta_0=0.$
		
		\noindent \textbf{(III)} We argue by contradiction and assume that $\beta_0 = 0$. Under this assumption, one can find a vanishing sequence $\beta_n \to 0^+$ satisfying
		\begin{equation}\label{d82}
			0 < m_{\beta_n}(\rho_1, \rho_2) < \frac{d}{N} \mathcal{S}(a, b)^{N/2d}, \quad \forall\, n \in \mathbb{N} .
		\end{equation}
		In view of Lemma~\ref{D9}, Lemma~\ref{D11}, and the compactness result in Proposition~\ref{D5} with energy level $c = m_{\beta_n}(\rho_1, \rho_2)$, the infimum $m_{\beta_n}(\rho_1, \rho_2)$ is attained by a pair $(u_n, v_n) \in \mathcal{N}_{\beta_n}(\rho_1, \rho_2)$. Hence, $(u_n, v_n)$  is a positive solution to the coupled system
		\begin{equation}\label{d83}
			\begin{cases}
				-\operatorname{div}(|x|^{-2a}\nabla u_n) + \lambda_{1,n} \dfrac{u_n}{|x|^{2a}} = \beta_n p \dfrac{|u_n|^{p-2}u_n |v_n|^q}{|x|^{b(p+q)}} + \dfrac{|u_n|^{2^\sharp-2}u_n}{|x|^{b2^\sharp}}, & \text{in } \mathbb{R}^N, \\
				-\operatorname{div}(|x|^{-2a}\nabla v_n) + \lambda_{2,n} \dfrac{v_n}{|x|^{2a}} = \beta_n q \dfrac{|u_n|^p |v_n|^{q-2}v_n}{|x|^{b(p+q)}} + \dfrac{|v_n|^{2^\sharp-2}v_n}{|x|^{b2^\sharp}}, & \text{in } \mathbb{R}^N,
			\end{cases} 
		\end{equation}
		associated with strictly positive multipliers $\lambda_{1,n}, \lambda_{2,n} > 0$ and mass constraints $\||x|^{-a}u_n\|_2 = \rho_1$, $\||x|^{-a}v_n\|_2 = \rho_2$.
		Using the Pohozaev identity $P_{\beta_n}(u_n, v_n) = 0$ on $\mathcal{N}_{\beta_n}(\rho_1, \rho_2)$, we decompose the energy level as follows
		\begin{equation}\label{d84}
			\begin{aligned}
				\frac{d}{N} \mathcal{S}(a, b)^{N/2d} > m_{\beta_n}(\rho_1, \rho_2) &= E_{\beta_n}(u_n, v_n) - \frac{1}{2}P_{\beta_n}(u_n, v_n) \\
				&= \frac{d}{N} \left( \||x|^{-b}u_n\|_{2^\sharp}^{2^\sharp} + \||x|^{-b}v_n\|_{2^\sharp}^{2^\sharp} \right) + \frac{(p+q)\delta_{p+q} - 2}{2} \beta_n \mathcal{B}_{p,q}(u_n, v_n).
			\end{aligned}
		\end{equation}
		Taking into account that $(p+q)\delta_{p+q} > 2$ and $\beta_n > 0$, relation \eqref{d84} ensures
		\begin{equation*}
			\limsup_{n \to \infty} \||x|^{-b}u_n\|_{2^\sharp}^{2^\sharp} < \mathcal{S}(a, b)^{N/2d} \quad \text{and} \quad \limsup_{n \to \infty} \||x|^{-b}v_n\|_{2^\sharp}^{2^\sharp} < \mathcal{S}(a, b)^{N/2d} .
		\end{equation*}
		Furthermore, since $p, q \ge 2$, H\"older's inequality together with the weighted Sobolev interpolation embedding leads to the bound
		\begin{equation*}
			\mathcal{B}_{p,q}(u_n, v_n) = \int_{\mathbb{R}^N} \frac{|u_n|^p |v_n|^q}{|x|^{b(p+q)}}\dx \le \||x|^{-b}u_n\|_{2^\sharp}^p \||x|^{-b}v_n\|_{\frac{2^\sharp q}{2^\sharp - p}}^q \le C_1(\rho_2) \||x|^{-b}u_n\|_{2^\sharp}^p ,
		\end{equation*}
		for some uniform constant $C_1(\rho_2) > 0$. By symmetry, we  also obtain $$\mathcal{B}_{p,q}(u_n, v_n) \le C_2(\rho_1) \||x|^{-b}v_n\|_{2^\sharp}^q.$$
		
		\noindent Testing \eqref{d83} against $u_n$ and $v_n$ respectively and applying the definition of the sharp Caffarelli--Kohn--Nirenberg constant $\mathcal{S}(a, b)$, we deduce
		\begin{align*}  
			\mathcal{S}(a, b)\||x|^{-b}u_n\|_{2^\sharp}^2 \le \||x|^{-a}\nabla u_n\|_2^2 &\le \||x|^{-b}u_n\|_{2^\sharp}^{2^\sharp} + \beta_n p C_1(\rho_2) \||x|^{-b}u_n\|_{2^\sharp}^p, \\
			\mathcal{S}(a, b)\||x|^{-b}v_n\|_{2^\sharp}^2 \le \||x|^{-a}\nabla v_n\|_2^2 &\le \||x|^{-b}v_n\|_{2^\sharp}^{2^\sharp} + \beta_n q C_2(\rho_1) \||x|^{-b}v_n\|_{2^\sharp}^q .
		\end{align*}
		Because $p \ge 2$, $q \ge 2$, and $2^\sharp > 2$, letting $n \to \infty$ with $\beta_n \to 0^+$ yields
		\begin{equation*}
			\liminf_{n \to \infty} \||x|^{-b}u_n\|_{2^\sharp}^{2^\sharp} \ge \mathcal{S}(a, b)^{N/2d} \quad \text{and} \quad \liminf_{n \to \infty} \||x|^{-b}v_n\|_{2^\sharp}^{2^\sharp} \ge \mathcal{S}(a, b)^{N/2d} .
		\end{equation*}
		Substituting these asymptotic bounds back into \eqref{d84}, we arrive at
		\begin{equation*}
        \begin{aligned}
			\liminf_{n \to \infty} m_{\beta_n}(\rho_1, \rho_2) \ge & \frac{d}{N}\left( \liminf_{n \to \infty} \||x|^{-b}u_n\|_{2^\sharp}^{2^\sharp} + \liminf_{n \to \infty} \||x|^{-b}v_n\|_{2^\sharp}^{2^\sharp} \right)\\
            & \ge \frac{2d}{N}\mathcal{S}(a, b)^{N/2d} > \frac{d}{N}\mathcal{S}(a, b)^{N/2d},
                 \end{aligned}
		\end{equation*}
		which contradicts \eqref{d82}. Consequently, we conclude that $\beta_0 > 0$ when $p, q \ge 2$. 
		\vspace{0.3cm}
		
		\noindent \textbf{(IV)} Let $\beta \in (0, \beta_0)$. By the threshold definition of $\beta_0$ combined with the monotonicity properties established in Lemma~\ref{D11} (ii) and part (I), we have the identity
		\begin{equation*}
			m_\beta(\rho_1, \rho_2) = \frac{d}{N}\mathcal{S}(a, b)^{N/2d}.
		\end{equation*}
		To establish non-attainability, suppose on the contrary that there exists a minimizer $(u, v) \in \mathcal{N}_\beta(\rho_1, \rho_2)$ satisfying $E_\beta(u, v) = m_\beta(\rho_1, \rho_2)$. Fix $\beta'$ with $\beta<\beta'<\beta_0$ and let $t':=t_{\beta'}^\u$ be given by Lemma~\ref{D6}, so that $t'\star\u\in\mathcal N_{\beta'}\r$. Since $\mathcal B_{p,q}\u>0$ and $\beta'>\beta$,
		\begin{align*}  
			{m_{\beta'}(\rho_1,\rho_2)}&{\le E_{\beta'}(t'\star\u)=E_\beta(t'\star\u)-(\beta'-\beta)e^{\pq t'}\mathcal B_{p,q}\u}\\
			&{<E_\beta(t'\star\u)\le\max_{t\in\R}E_\beta(t\star\u)=E_\beta\u=\frac dN\mathcal S(a,b)^{N/2d},}
		\end{align*}
		the last equality of the second line by Lemma~\ref{D6}(iii), since $\u\in\mathcal N_\beta\r$. Thus $m_{\beta'}\r<\frac dN\mathcal S(a,b)^{N/2d}$, so $\beta'\ge\beta_0$ by the definition of $\beta_0$, contradicting $\beta'<\beta_0$. Hence, $m_\beta(\rho_1, \rho_2)$ cannot be achieved for any $0 < \beta < \beta_0$. \qed
		
	\end{proof}
\begin{rem}\label{D304}
Let $p+q>\p$ and suppose $\beta_0>0$. Then $m_{\beta_0}\r=\frac dN\sr^{N/2d}$. Indeed, for
$\u\in S\r$ the map $\beta\mapsto\max_{t\in\R}E_\beta(t\star\u)$ is a supremum of affine functions of $\beta$, hence convex and continuous on $(0,+\infty)$. By Lemma~\ref{D7},
$m_\beta\r$ is an infimum of such maps, so $\limsup_{\beta\to\beta_0}m_\beta\r\le m_{\beta_0}\r$.
Since $m_\beta\r=\frac dN\sr^{N/2d}$ for $0<\beta<\beta_0$ by Lemma~\ref{D201}(IV), this gives $m_{\beta_0}\r\ge\frac dN\sr^{N/2d}$, and the reverse inequality is Lemma~\ref{D201}(I). Thus, a ground state at $\beta=\beta_0$, if one exists, would have energy exactly at the compactness
threshold of Proposition~\ref{D5}. Whether $m_{\beta_0}\r$ is attained is an open question.
\end{rem}
	\noi \textbf{Proof of Theorem~\ref{DD6}:} Combining Proposition~\ref{D5}, Lemma~\ref{D9}, Lemma~\ref{D11}, and Lemma~\ref{D201} yields the desired result; the endpoint $\beta=\beta_0$ is discussed in Remark~\ref{D304}. \qed

	\subsection{Mass-critical case}
	Let 
	\begin{equation}\label{d222}
		\beta^* := \frac{1}{2 C_{a,b}^{p+q} (\rho_1^2 + \rho_2^2)^{\frac{p+q-2}{2}}}.
	\end{equation}
	
	In this subsection, we consider the mass-critical case $p+q = (p+q)_c$, corresponding to $(p+q)\delta_{p+q} = 2$. To study the variational structure for $0 < \beta < \beta^*$, we first introduce and investigate the subset $\mathcal{T}_\beta(\rho_1, \rho_2) \subset  S\r$. 
	\begin{lemma}\label{D300}
		For $\beta>0$, define 
		\begin{align*}  \mathcal{T}_\beta\r:=\Big\{\u\in S\r:\||x|^{-a}\na u\|_2^2+\||x|^{-a}\na v\|_2^2>2\beta\b\u\Big\}.\end{align*}
		Then $\mathcal{T}_\beta\r\ne\emptyset.$
		Moreover, if $0<\beta<\beta^*$, then $\mathcal{T}_\beta\r=S\r.$
	\end{lemma}
	\begin{proof}
		Let $0<\beta<\beta^*$ and  $\u\in S\r$. By inequality \eqref{d5}, we have
		\begin{align*}  
			2\beta\b\u\le 2\beta C_{a,b}^{p+q}(\rho_1^2+\rho_2^2)^{\frac{p+q-2}{2}}\left(\||x|^{-a}\na u\|_2^2+\||x|^{-a}\na v\|_2^2\right)<\||x|^{-a}\na u\|_2^2+\||x|^{-a}\na v\|_2^2.
		\end{align*}
		Therefore,  $\mathcal{T}_\beta\r=S\r$. Now assume $\beta\ge\beta^*$. To show that $\mathcal{T}_\beta\r\ne\emptyset, $ it suffices to prove that
		\begin{equation*}
			\sup_{\u\in S\r}\frac{\||x|^{-a}\na u\|_2^2+\||x|^{-a}\na v\|_2^2}{\b\u}=+\infty.
		\end{equation*}
		Define
		\begin{equation*}
			\phi_n(x):=\begin{cases}
				\sin n|x|,\quad&\text{if}~|x|\le\pi,\\
				0,&\text{if}~|x|>\pi.
			\end{cases}
		\end{equation*}
		Set ${(u_n,v_n)}:=\Big(\frac{\rho_1}{\||x|^{-a}\phi{_n}\|_2}\phi_n,\frac{\rho_2}{\||x|^{-a}\phi{_n}\|_2}\phi_n\Big)$, which clearly belongs to ${(u_n,v_n)}\in S\r$. Since ${\||x|^{-a}\na\phi_n\|_2^2}=O(n^2)$, $\||x|^{-a}\phi{_n}\|_2^2=O(1)$, and $\||x|^{-b}\phi_n\|_{p+q}^{p+q}=O(1)$, it follows that 
		\begin{align*}  \frac{\||x|^{-a}\na {u_n}\|_2^2+\||x|^{-a}\na {v_n}\|_2^2}{\b{(u_n,v_n)}}=\frac{(\rho_1^2+\rho_2^2){\||x|^{-a}\na\phi_n\|_2^2}}{\rho_1^p\rho_2^q\||x|^{-b}\phi_n\|_{p+q}^{p+q}{\||x|^{-a}\phi_n\|_2^{p+q-2}}}\to+\infty,\quad\text{as}~n\to\infty.\end{align*}
		The lemma follows.\qed
	\end{proof}
	\begin{lemma}\label{D12}
		For $\u\in \mathcal{T}_\beta\r$, the fibering map $\Phi_\beta^{\u}$ admits exactly one critical point 
		$t_\beta^{\u}$. 
		Moreover, the following properties hold:
		\begin{itemize}
			\item[(i)] $\mathcal{N}_\beta\r=\mathcal{N}^-_\beta\r$ and $\mathcal{N}_\beta{\r} $ is a submanifold of $\y$.
			\item[(ii)] $t\star\u\in \mathcal{N}_\beta\r$ if and only if $t=t_\beta^\u.$
			\item[(iii)] $\Phi_\beta^\u(t)$ is strictly decreasing and concave on $(t_\beta^\u,+\infty)$ and 
			\begin{align*}  \Phi_\beta^\u(t_\beta^\u)=\max_{t\in\R}\Phi_\beta^\u(t)>0.\end{align*}
			\item[(iv)] The map $\u\mapsto t_\beta^\u$ is of class $C^1$.
		\end{itemize}
	\end{lemma}

	\begin{proof}
		The proof follows along the same lines as that of Lemma~\ref{D6}, relying on the monotonicity properties of the associated fibering map and the implicit function theorem.\qed
	\end{proof}
	\begin{lemma}\label{D301}
		$\displaystyle m_\beta\r=\inf_{\mathcal{T}_\beta\r}\max_{t\in\R}E_\beta(t\star\u)$.
	\end{lemma}
	\begin{proof}
		The statement follows immediately from Lemma~\ref{D12}.  \qed
	\end{proof}
	\begin{lemma}\label{D91}
		If $0<\beta<\beta^*$, then the functional 
		$E_\beta\big|_{S\r}$ admits a Palais–Smale sequence 
		$\{\v\}\subset S\r$ at the level $m_\beta\r$ such that
		\begin{align*}  
			P_\beta\v\to 0,
			\qquad
			u_n^{-},\, v_n^{-}\to 0 
			\quad \text{a.e. in } \R^N
			\quad \text{as } n\to\infty.
		\end{align*}
	\end{lemma}
	\begin{proof}
		The argument follows along the same lines as in Lemma~\ref{D9}.\qed
	\end{proof}
	\begin{lemma}\label{D92}
		Let $0<\beta<\beta^*$. For any fixed $\rho_1,~\rho_2>0$, the following hold
		\begin{itemize}
			\item[(i)] $m_\beta\r\le m_\beta\rt$ for any $0<\tilde{\rho}_1<\rho_1$ and $0<\tilde{\rho}_2{<\rho_2}.$
			\item[(ii)] $m_\beta\r$ is nonincreasing with respect to $\beta\in(0,\beta^*).$
		\end{itemize}
	\end{lemma}
	\begin{proof}
		\textit{(i)} Since $0<\beta<\beta^*$, Lemma~\ref{D300} gives $\mathcal{T}_\beta\r=S\r$, and  the inequality follows by the same argument as in Lemma~\ref{D11}, with Lemma~\ref{D12}(iv) in place of Lemma~\ref{D6}\textit{(iv)}.
		
		\textit{(ii)} For any $0<\beta'\le\beta<\beta^*$, we have
		\begin{align*}  m_\beta\r=\inf_{S\r}\max_{t\in\R}E_\beta(t\star\u)\le \inf_{S\r}\max_{t\in\R}E_{\beta'}(t\star\u)=m_{\beta'}\r.\end{align*}
		Therefore, $m_\beta\r$ is nonincreasing with respect to $\beta\in(0,\beta^*).$\qed
	\end{proof}
	
	\begin{lemma}\label{D93}
		Assume that $0 < \beta < \beta^*$. Set $r := \min\{p, q\}$. Then $$0 < m_\beta(\rho_1, \rho_2) < \frac{d}{N} {\mathcal S}(a, b)^{N/2d} ,$$ 
		provided one of the following conditions is satisfied
		\begin{itemize}
			\item[(i)] $a = \frac{N-4}{2} > 0$, and $1 < r < 2 - \frac{4d(4+2a-N)}{(N-2-2a)(N-2a+2b)}$.
			\item[(ii)] $\max\left\{0, \frac{N-4}{2}\right\} < a < \frac{N-2}{2}$, and $$\max\left\{ \frac{N}{N - 2(1+a) + b}, \, 2^\sharp - \frac{2(2 - r)(N - 2a - 2)}{4 + 2a - N} \right\} < p + q < 2^\sharp.$$
			\item[(iii)] $\max\left\{0, \frac{N-4}{2}\right\} < a < \frac{N-2}{2}$ $($with $b - a > \frac{1}{6}$ when $N = 3)$, and$$p + q \le \frac{N}{N - 2(1+a) + b},$$along with$$1 < r < 2 - \frac{N - 4 + 6(b-a)}{2(N - 2a + 2b)}.$$
		\end{itemize}
	\end{lemma}
	\begin{proof}We first prove the positivity of $m_\beta(\rho_1,\rho_2)$.
		For $\u\in\mathcal{N}_\beta\r$ there holds
		\begin{align*}  
			\||x|^{-a}\na u\|_2^2+\||x|^{-a}\na v\|_2^2&=\||x|^{-b}u\|_\t^\t+\||x|^{-b}v\|_\t^\t+2\beta\b\u\\
			&\le \||x|^{-b}u\|_\t^\t+\||x|^{-b}v\|_\t^\t+\frac{\beta}{\beta^*}\left(\||x|^{-a}\na u\|_2^2+\||x|^{-a}\na v\|_2^2 \right)\\
			&\le \mathcal{S}(a,b)^{{-\t/2}} \left(\||x|^{-a}\na u\|_2^2+\||x|^{-a}\na v\|_2^2\right)^{\t/2}\\
            & ~~ ~~+\frac{\beta}{\beta^*}\left(\||x|^{-a}\na u\|_2^2+\||x|^{-a}\na v\|_2^2\right).
		\end{align*}
		Hence
		\begin{align*}  
			\inf_{(u,v)\in {\mathcal N}_\beta(\rho_1,\rho_2)}
			\left(
			\||x|^{-a}\nabla u\|_2^2
			+\||x|^{-a}\nabla v\|_2^2
			\right)>0,
		\end{align*}
		and therefore
		
		\begin{align*}  m_\beta\r=\inf_{\mathcal{N}_\beta\r}\frac{d}{N}\left(\||x|^{-b}u\|_\t^\t+\||x|^{-b}v\|_\t^\t\right)>0.\end{align*}
		Since the functional $E_\beta$ and the constraint
		$S(\rho_1,\rho_2)$ are invariant under the interchange
		$
		(p,u,\rho_1)$ and $(q,v,\rho_2),
		$
		we may assume, without loss of generality, that
		$
		p>q.
		$
		In this case, $r=q$. The case $q>p$ follows by interchanging
		the two components.
		
		Let $\zeta_\varepsilon=\varphi U_\varepsilon$ be the test
		function introduced in Lemma~\ref{D19}. For $\theta>0$, define
		\begin{align*}  
			u_\varepsilon
			=
			\frac{\rho_1}
			{\||x|^{-a}\zeta_\varepsilon\|_2}\zeta_\varepsilon,
			\qquad
			v_\varepsilon
			=
			\frac{\varepsilon^\theta}
			{\||x|^{-a}\zeta_\varepsilon\|_2}\zeta_\varepsilon .
		\end{align*}
		Then
		\begin{align*}  
			\||x|^{-a}u_\varepsilon\|_2=\rho_1,
			\qquad
			\||x|^{-a}v_\varepsilon\|_2=\varepsilon^\theta.
		\end{align*}
		
		Hence
		$
		(u_\varepsilon,v_\varepsilon)
		\in S(\rho_1,\varepsilon^\theta).
		$
		For $\varepsilon>0$ sufficiently small,
		$\varepsilon^\theta<\rho_2$, and therefore Lemma~{\ref{D92}}~$(i)$
		gives
		$
		m_\beta(\rho_1,\rho_2)
		\leq m_\beta(\rho_1,\varepsilon^\theta).
		$Let
		$
		t_\varepsilon
		=
		t_\beta^{(u_\varepsilon,v_\varepsilon)}
		$
		be given by Lemma~{\ref{D12}}; this applies because $\beta<\beta^*(\rho_1,\rho_2)\le\beta^*(\rho_1,\var^\theta)$, so that $(u_\var,v_\var)\in\mathcal T_\beta(\rho_1,\var^\theta)$ by Lemma~\ref{D300}. Then
		$
		t_\varepsilon\star
		(u_\varepsilon,v_\varepsilon)
		\in
		\mathcal N_\beta(\rho_1,\varepsilon^\theta),
		$
		and consequently
		\begin{align*}  
			m_\beta(\rho_1,\rho_2)
			\leq
			E_\beta
			\left(
			t_\varepsilon\star
			(u_\varepsilon,v_\varepsilon)
			\right).
		\end{align*}
		Using the definition of $E_\beta$ and the explicit form of
		$u_\varepsilon,v_\varepsilon$, we obtain
		\begin{align*}  
			\begin{aligned}
				m_\beta(\rho_1,\rho_2)
				\leq {}&
				\frac dN \mathcal{S}(a,b)^{\frac N{2d}}
				+O\left(\varepsilon^{\frac{N-2d}{2d}}\right)
				+O(\varepsilon^{2\theta})
				\\
				&-
				C\varepsilon^{q\theta}
				e^{(p+q)\delta_{p+q}t_\varepsilon}
				\frac{
					\||x|^{-b}\zeta_\varepsilon\|_{p+q}^{p+q}
				}{
					\||x|^{-a}\zeta_\varepsilon\|_2^{p+q}
				}.
			\end{aligned}
		\end{align*}
		Here we have used that $P_\beta(t_\var\star(u_\var,v_\var))=0$ and $(u_\var,v_\var)\in\mathcal T_\beta(\rho_1,\var^\theta)$ give
			\begin{align*}  
				e^{(\t-2)t_\var}=\frac{\||x|^{-a}\na u_\var\|_2^2+\||x|^{-a}\na v_\var\|_2^2-2\beta\b(u_\var,v_\var)}{\||x|^{-b}u_\var\|_\t^\t+\||x|^{-b}v_\var\|_\t^\t}\ge\Big(1-\frac{\beta}{\beta^*}\Big)\frac{\||x|^{-a}\na u_\var\|_2^2+\||x|^{-a}\na v_\var\|_2^2}{\||x|^{-b}u_\var\|_\t^\t+\||x|^{-b}v_\var\|_\t^\t},
			\end{align*}
			hence $e^{t_\var}\ge C\||x|^{-a}\zeta_\var\|_2$ by Lemma~\ref{D19}.

        \noi
		\textit{(i)}   If  $\frac{N}{N-2(1+a)+b}<(p+q)_c$, then 
		\begin{equation*}
			\frac{\||x|^{-b}\zeta_\var\|_{(p+q)_c}^{(p+q)_c}}{\||x|^{-a}\zeta_\var\|_2^{\frac{4d}{N-2a+2b}}}= \begin{cases}
				C|\log \var|^{-\frac{2d}{N-2a+2b}},&\text{if}~0<a=\frac{N-4}{2},\\
				C\var^{\frac{(N-2d)(4+2a-N)}{(N-2-2a)(N-2a+2b)}},&\text{if}~\max\{0,\frac{N-4}{2}\}<a<\frac{N-2}{2}.
			\end{cases}
		\end{equation*}
		and we choose \begin{align*}  \theta\in\left(\frac{(N-2d)(4+2a-N)}{(2-q)(N-2-2a)(N-2a+2b)},\frac{(N-2d)}{q}\left(\frac{1}{2d}-\frac{4+2a-N}{(N-2-2a)(N-2a+2b)}\right)\right)\end{align*}
		and $1<q<2-\frac{4d(4+2a-N)}{(N-2-2a)(N-2a+2b)}.$
		
		\noi
		\textit{(ii)} If $  (p+q)_c \leq \frac{N}{N-2(1+a)+b}$, then 
		\begin{equation*}
			\frac{\||x|^{-b}\zeta_\var\|_{(p+q)_c}^{(p+q)_c}}{\||x|^{-a}\zeta_{\var}\|_2^{\frac{4d}{N-2a+2b}}}=\begin{cases}
				C\var^{\frac{(N-2d)(N-4+6(b-a))}{4d(N-2a+2b)}}|\log \var|,\quad &\text{if}~~{(p+q)_c}= \frac{N}{N-2(1+a)+b},\\
				C\var^{\frac{(N-2d)(N-4+6(b-a))}{4d(N-2a+2b)}}, &\text{if}~~{(p+q)_c}<\frac{N}{N-2(1+a)+b}.
			\end{cases} 
		\end{equation*}
		Taking $1 < q < 2 - \frac{N-4+6(b-a)}{2(N-2a+2b)}$, we choose  $$\theta \in \left( \frac{(N-2d)(N-4+6(b-a))}{4d(2-q)(N-2a+2b)}, \; \frac{N-2d}{2dq}\left( 1 - \frac{N-4+6(b-a)}{2(N-2a+2b)} \right) \right).$$As a consequence, for $\varepsilon > 0$ sufficiently small, we obtain  $$m_\beta(\rho_1, \rho_2) < \frac{d}{N} {\mathcal S}(a, b)^{N/2d} .$$
		
		By choosing $\theta > 0$ in the stated range in  either case, we obtain for $\var > 0$ sufficiently small $$m_\beta(\rho_1, \rho_2) < \frac{d}{N}\mathcal{S}(a,b)^{N/2d}.$$ This completes the proof.\qed
	\end{proof}
	\noi \textbf{Proof of Theorem~\ref{DD9}:} The result follows immediately by combining Proposition~\ref{D5} with Lemmas~\ref{D91}, \ref{D92}, and \ref{D93}. \qed
	\section{Nonexistence of ground states for  \texorpdfstring{${\beta\le0}$}{beta <= 0}}\label{SD3}
	In this section, we show that for $\beta\le0$ the ground state level $m_\beta\r$ is not attained.
	\begin{lemma}\label{D78}
		Let $\beta\le0$ and $\u\in S\r$. Then there exists a unique $t_\beta^\u\in\R$ such that $t_\beta^\u\star\u\in\mathcal{N}_\beta\r$.  Moreover, $t_\beta^{\u}$ is the unique critical point of the fibering map
		$\Phi_\beta^{\u}$ and is a strict maximum point, at a positive level. In addition, the following properties hold
		\begin{itemize}
			\item[(i)] $\mathcal{N}_\beta\r=\mathcal{N}_\beta^-\r$.
			\item[(ii)] $t_\beta^\u<0$ if and only if $P_\beta\u<0$.
			\item[(iii)] $\Phi_\beta^\u$ is strictly decreasing and concave on $(t_\beta^\u,+\infty)$.
			\item[(iv)] The map $\u\mapsto t_\beta^\u$ is of class $C^1.$ 
		\end{itemize}
	\end{lemma}
	\begin{proof}
		Notice that
		\begin{equation*}\begin{aligned}
				(\Phi_\beta^\u)'(t)=e^{\t t}\left(\frac{1}{e^{(\t-2)t}}\left(\||x|^{-a}\na u\|_2^2+\||x|^{-a}\na v\|_2^2\right)+
				\frac{|\beta|\pq}{e^{(\t-\pq)t}}\b\u \right. \\
				\quad \left.-\left(\||x|^{-b}u\|_\t^\t+\||x|^{-b}v\|_\t^\t\right)\right)
			\end{aligned}
		\end{equation*}
		is strictly decreasing because $\pq<\t$. It follows that $(\Phi_\beta^\u)'$ has a unique zero $t_\beta^\u$, and $t_\beta^\u\star\u\in\mathcal{N}_\beta\r.$ Moreover, $(\Phi_\beta^\u)'(t)>0$ for $t<t_\beta^\u $, and  $(\Phi_\beta^\u)'(t)<0$ for $t>t_\beta^\u $,  so $t_\beta^\u$ is a strict maximum point of $\Phi_\beta^\u$, and $\Phi_\beta^\u(t_\beta^\u)>0$ since $\Phi_\beta^\u(t)>0$ for $t\to-\infty$. By these properties of $\Phi_\beta^\u$, conclusions  \textit{(i)--(iv)} follow.\qed
	\end{proof}
	Define
	\begin{align*}  \mathcal{L}_\beta\r:=\{\u\in S\r: P_\beta\u<0\},\end{align*}
	\begin{align*}  {\mathcal{K}}_\beta\r:=\{\u\in S\r: P_\beta\u\le0\}\end{align*}
	and
	\begin{align*}  
		F_\beta\u:&=E_\beta\u-\frac{1}{\t}P_\beta\u\\
		&=\frac{d}{N}\left(\||x|^{-a}\na u\|_2^2+\||x|^{-a}\na v\|_2^2\right)+\frac{\t-\pq}{\t}|\beta|\b\u.
	\end{align*}
	\begin{lemma}
		Let $\beta\le0$ and $\rho_1,~\rho_2>0$. Then the following holds.
		\begin{itemize}
			\item[(i)] $\displaystyle m_\beta\r=\inf_{{\mathcal{K}_\beta}\r}F_\beta\u=\inf_{\mathcal{L}_\beta\r}F_\beta\u.$
			\item[(ii)] If $0<\tilde{\rho}_1\le \rho_1,~0<\tilde{\rho}_2<\rho_2$ then $m_\beta\r\le m_\beta(\tilde{\rho}_1,\tilde{\rho}_2).$
			\item[(iii)] For each $\beta\le 0$ there holds $m_\beta\r=\frac{d}{N}\mathcal{S}(a,b)^{N/2d}$. 
		\end{itemize}
	\end{lemma}
	\begin{proof} 
		\textit{(i)} By Lemma~\ref{D78}, $\mathcal{L}{_\beta}\r$ is dense in $\mathcal{K}{_\beta}\r$.  Moreover, $\mathcal{N}_\beta\r\subset\mathcal{K}{_\beta}\r$. Therefore
		\begin{align*}  \inf_{{\mathcal{K}}_\beta\r}{F_\beta}\u=\inf_{\mathcal{L}_\beta\r}{F_\beta}\u\le \inf_{\mathcal{N}_\beta\r}{F_\beta}=\inf_{\mathcal{N}_\beta\r}{E_\beta}=m_\beta\r.\end{align*}
		Next, let $\u\in \mathcal{L}{_\beta}\r$. By Lemma~\ref{D78}, we have $t_\beta^\u<0$.  Hence, since $e^{2t_\beta^\u}<1$ and $e^{\pq t_\beta^\u}<1$,
		\begin{align*}  
			F_\beta\u&>\frac{d}{N}e^{2t_\beta^\u}\left(\||x|^{-a}\na u\|_2^2+\||x|^{-a}\na v\|_2^2\right)+\frac{\t-\pq}{\t}|\beta|e^{t_{\beta}^\u\pq}\b\u\\
			&=F_\beta(t_\beta^\u\star\u)=E_\beta(t_\beta^\u\star\u)\ge \inf_{\mathcal{N}_\beta\r}E_\beta.
		\end{align*}
		Therefore, $\displaystyle\inf_{\mathcal{L}_\beta\r}F_\beta\ge \inf_{\mathcal{N}_\beta\r}E_\beta$.\\
		\textit{(ii)} Using the same argument as in Lemma~\ref{D11}, one can prove that \textit{(ii)} holds.\\
		\textit{(iii)}  First we show that $m_\beta\r= m_0\r$ for $\beta<0.$ Since $P_0\u\le P_\beta\u$ and $F_0\u\le F_\beta\u$ we have $\mathcal{K}_\beta{\r}\subset \mathcal{K}_0{\r}$ and
		\begin{align*}  {m_0\r}=\inf_{\mathcal{K}_0\r}{F_0}\u\le\inf_{\mathcal{K}_\beta\r}{F_\beta}\u=m_\beta{\r}.\end{align*}
		On the other hand, for $\u\in\mathcal{L}_0\r$ and $t>0$ we set $$\left( u_t(x),v_t(x)\right)=\left(t^{\frac{N-2a-2}{2}}u(tx),t^{\frac{N-2a-2}{2}}v(tx)\right)$$
		and obtain for large $t>1$
		\begin{align*}  
			P_\beta(u_t,v_t)&=\||x|^{-a}\na u\|_2^2+\||x|^{-a}\na v\|_2^2-\||x|^{-b}u\|_\t^\t-\||x|^{-b}v\|_\t^\t+\frac{\beta\pq}{t^{(1-\delta_{p+q})(p+q)}}\b\u\\
			&=P_0\u+\frac{\beta\pq}{t^{(1-\delta_{p+q})(p+q)}}\b\u<0.
		\end{align*}
		Observing that $||x|^{-a}u_t\|_2=t^{-1}\rho_1,~\||x|^{-a}v_t\|_2=t^{-1}\rho_2,$ we get for $t>1$
		\begin{align*}  
			m_\beta\r&\le m_\beta(t^{-1}\rho_1,t^{-1}\rho_2)\le {F_\beta}(u_t,v_t)\\
			&=\frac{d}{N}\left(\||x|^{-a}\na u\|_2^2+\||x|^{-a}\na v\|_2^2\right)+\frac{\t-\pq}{\t}|\beta|\b(u_t,v_t)\\
			&= {F_0}\u+\frac{\t-\pq}{\t t^{(1-\delta_{p+q})(p+q)}}|\beta|\b\u\to~{F_0}\u\quad\textit{as}~t\to\infty.
		\end{align*}
		This proves our claim. Next we  show that $m_\beta\r=\frac{d}{N}\sr^{N/2d}.$
		For $\u\in\mathcal{N}_0\r,$ we have
		\begin{align*}  
			\||x|^{-a}\na u\|_2^2+\||x|^{-a}\na v\|_2^2&=\||x|^{-b}u\|_\t^\t+\||x|^{-b}v\|_\t^\t\\
			&\le\sr^{-\t/2}\left(\||x|^{-a}\na u\|_2^\t+\||x|^{-a}\na v\|_2^\t\right)\\
			&\le\sr^{-\t/2}\left(\||x|^{-a}\na u\|_2^2+\||x|^{-a}\na v\|_2^2\right)^{\t/2}.    \end{align*}
		and therefore $m_0\r\ge \frac{d}{N}\sr^{N/2d}$. Arguing as in \cite[Proposition 2.2]{soave2020normalized2},  for every $\var>0$ there exists $u\in X$ with $\||x|^{-a}u\|_2=\rho_1$ such that
		\begin{equation*}
			\max_{t\in \R} I(t\star u)\le \frac{d}{N}\sr^{N/2d}+\var.
		\end{equation*}
		Now, we choose any $v\in X$ with $\||x|^{-a}v\|_2=\rho_2,$ and define $t(k):=t_0^{(u,k\star v)}$. Observe that $\displaystyle\limsup_{{k}\to-\infty}t(k)<+\infty$ as in  the proof of Lemma~\ref{D11}~\textit{(i)}. We obtain
		\begin{align*}  
			m_0\r&\le E_0(t(k)\star (u,k\star v))=I(t(k)\star u)+I((t(k)+k)\star v)\\
			&\le\frac{d}{N}\sr^{N/2d}+\var+I((t(k)+k)\star v)\to \frac{d}{N}\sr^{N/2d}+\var\quad\text{as}~k\to-\infty,
		\end{align*}
		since $t(k)+k\to-\infty$ and $I(s\star v)\to0$ as $s\to-\infty$.
		Since this holds for any $\var>0,$ we deduce $m_0\r\le \frac {d}{N}\sr^{N/2d}$.\qed
	\end{proof}
	\noindent
	\textbf{Proof of Theorem~\ref{DD1}} For $\u \in\mathcal{N}_\beta\r$ we have for $\beta\le 0$
	\begin{align*}  
		\||x|^{-a}\na u\|_2^2+\||x|^{-a}\na v\|_2^2&=\||x|^{-b}u\|_\t^\t+\||x|^{-b}v\|_\t^\t+\pq\beta\b\u\\
		&\le \||x|^{-b}u\|_\t^\t+\||x|^{-b}v\|_\t^\t\\
		&\le \sr^{-\t/2}\left( \||x|^{-a}\na u\|_2^\t+\||x|^{-a}\na v\|_2^\t \right)\\
		&\le \sr^{-\t/2}\left( \||x|^{-a}\na u\|_2^2+\||x|^{-a}\na v\|_2^2 \right)^{\t/2}
	\end{align*}
	so that $\left( \||x|^{-a}\na u\|_2^2+\||x|^{-a}\na v\|_2^2\right)\ge\sr^{N/2d}.$ This implies for $\beta\le 0$
	\begin{align*}  
		E_\beta\u&= E_\beta\u-\frac{1}{\t}P_\beta\u
		\\&=\frac{d}{N} \left( \||x|^{-a}\na u\|_2^2+\||x|^{-a}\na v\|_2^2\right)-\frac{\t-\pq}{\t}\beta\b\u\\
		&\ge\frac{d}{N} \left( \||x|^{-a}\na u\|_2^2+\||x|^{-a}\na v\|_2^2\right)\ge\frac{d}{N}\sr^{N/2d}=m_\beta\r.
	\end{align*}
	It remains to exclude equality. If $E_\beta\u=m_\beta\r$ then every inequality above is an equality, so $\beta\mathcal B_{p,q}\u=0$ and $\||x|^{-a}\na u\|_2^2+\||x|^{-a}\na v\|_2^2=\mathcal S(a,b)^{N/2d}$ with equality in \eqref{d191}. Equality in the elementary inequality $x^{\t/2}+y^{\t/2}\le(x+y)^{\t/2}$ forces one of the two components to vanish, and the other is then an extremal of \eqref{d191}, hence equal to some $U_\var$. Since $U_\var\notin L^2(\R^N;|x|^{-2a})$, this contradicts $\u\in S\r$. Therefore $m_\beta\r$ is not attained.\qed
	\section{Ground states of the limit system}\label{SD4}
	In this section, we study the system \eqref{d60}, which has no critical term, and prove Theorem~\ref{DD4}. Following the approach in~\cite{szulkin2009ground}, for $(u, v) \in  S\r$ we define the limit functional
	\begin{equation*}
		J(u, v) := \frac{1}{2} \left( \||x|^{-a}\nabla u\|_2^2 + \||x|^{-a}\nabla v\|_2^2 \right) - \mathcal{B}_{p,q}(u, v),
	\end{equation*}
	and the corresponding limit Pohozaev manifold
	\begin{equation*}
		\mathcal{W}(\rho_1, \rho_2) := \left\{ (u, v) \in S\r : \||x|^{-a}\nabla u\|_2^2 + \||x|^{-a}\nabla v\|_2^2 = (p+q)\delta_{p+q} \mathcal{B}_{p,q}(u, v) \right\}.
	\end{equation*}
	For any $(u, v) \in  S\r$ with $\b\u>0$, the equation
	\begin{equation*}
		\frac{d}{dt} J(t \star (u, v)) = \frac{d}{dt} \left( \frac{e^{2t}}{2}\left( \||x|^{-a}\nabla u\|_2^2 + \||x|^{-a}\nabla v\|_2^2 \right) - e^{(p+q)\delta_{p+q}t} \mathcal{B}_{p,q}(u, v) \right) = 0
	\end{equation*}
	admits a unique solution $t^{(u,v)} \in \mathbb{R}$ given by
	\begin{equation*}
		e^{t^{(u,v)}} = \left( \frac{\||x|^{-a}\nabla u\|_2^2 + \||x|^{-a}\nabla v\|_2^2}{(p+q)\delta_{p+q} \mathcal{B}_{p,q}(u, v)} \right)^{\frac{1}{(p+q)\delta_{p+q} - 2}}.
	\end{equation*}
	Clearly, $t^{(u,v)} \star (u, v) \in \mathcal{W}(\rho_1, \rho_2)$. We now define the reduced functional $\widetilde{J} :  S\r \to \mathbb{R}$ by
	\begin{equation*}
		\widetilde{J}(u, v) := J(t^{(u,v)} \star (u, v)) = \frac{(p+q)\delta_{p+q} - 2}{2(p+q)\delta_{p+q}} \left( \frac{\left(\||x|^{-a}\nabla u\|_2^2 + \||x|^{-a}\nabla v\|_2^2\right)^{(p+q)\delta_{p+q}}}{\left((p+q)\delta_{p+q} \mathcal{B}_{p,q}(u, v)\right)^2} \right)^{\frac{1}{(p+q)\delta_{p+q} - 2}}.
	\end{equation*}
	Moreover, the ground state level \begin{equation}\label{d61}
		\kappa(\rho_1, \rho_2) := \displaystyle\inf_{\mathcal{W}(\rho_1, \rho_2)} J(u, v)
	\end{equation}  is the quantity of interest. We write
	\begin{align*}  \mathcal Z\r:=\{\u\in\mathcal W\r:J\u=\kappa\r\}\end{align*}
	for the set of ground states of \eqref{d60}. The level $\kappa\r$ has the following properties.
	
	\begin{lemma}\label{D105}
		$\kappa\r=\displaystyle\inf_{S\r}\widetilde{J}$ for all $\rho_1,\rho_2>0.$
	\end{lemma}
	\begin{proof}
		For $\u\in S\r$ we have $\displaystyle\widetilde{J}\u\ge\inf_{\mathcal{W}\r}J =\kappa\r$. On the other hand
		\begin{align*}  \inf_{S\r}\widetilde{J}\u\le\inf_{\mathcal W\r}\widetilde J\u=\inf_{\mathcal W\r}J\u= \kappa\r,\end{align*}
		where in the first inequality we use $\mathcal{W}\r\subset S\r$, and the equalities hold because $t^{(u,v)}=0$ on $\mathcal W\r$.
		\qed
	\end{proof}
	\begin{lemma}\label{D106}
		\begin{itemize}
			\item[(i)] If $p+q<(p+q)_c$ then $\kappa\r<0.$
			\item[(ii)] For $p+q>(p+q)_c$ then $\kappa\r>0.$
			\item[(iii)] $\kappa\r\le \kappa(\tilde{\rho}_1,\tilde{\rho}_2)$ for $0<\tilde{\rho_1}<\rho_1$ and $0<\tilde{\rho}_2<\rho_2.$
		\end{itemize}
	\end{lemma}
	\begin{proof}
		\textit{(i)} If $p+q<(p+q)_c$ then  $\pq<2$, and hence $\widetilde{J}\u<0$ for any $\u\in S\r$ with $\b\u>0$.

        \noi
		\textit{(ii)} 
		By \eqref{d5}, we have for $\u\in \mathcal{W}\r$
		\begin{align*} 
        \||x|^{-a}\na u\|_2^2+\||x|^{-a}\na v\|_2^2=& \pq\b\u \\
        & \le \pq C\left(\||x|^{-a}\na u\|_2^2+\||x|^{-a}\na v\|_2^2\right)^{\frac{\pq}{2}},
        \end{align*}
		which implies $\displaystyle\inf_{\mathcal{W}\r}\||x|^{-a}\na u\|_2^2+\||x|^{-a}\na v\|_2^2>0,$ and therefore
		\begin{align*}  
			\kappa\r=\inf_{\mathcal{W}\r}\left(\frac{1}{2}-\frac{1}{\pq}\right)\left(\||x|^{-a}\na u\|_2^2+\||x|^{-a}\na v\|_2^2\right)>0.\end{align*}
            \noi
		\textit{(iii)} If $p+q<(p+q)_c$,  arguing as in \cite[Lemma 3.1]{gou2016existence}, for $0<\tilde{\rho}_1<\rho_1$ and $0<\tilde{\rho}_2<\rho_2$, we get
		\begin{align*}  
			\kappa(\rho_1,\rho_2)\le \kappa(\tilde{\rho}_1,\tilde{\rho}_2)+\kappa\left(\sqrt{\rho^2_1-\tilde{\rho}^2_1},\sqrt{\rho_2^2-\tilde{\rho}_2^2}\right)<\kappa(\tilde{\rho}_1,\tilde{\rho}_2).\end{align*}
		For $p+q>\p$, the argument of Lemma~\ref{D11}\textit{(i)} applies.\qed
	\end{proof}
	\noi
	\textbf{ Proof of Theorem~\ref{DD4}} If  $p+q<(p+q)_c$ we obtain a Palais-Smale sequence by starting with a minimizing sequence $J\v\to\kappa\r$ with $\v\in S\r$. By Ekeland's variational principle, we may assume that $\v$ is a Palais-Smale sequence. Then we have
	\begin{align*}  
		J'\v+{\left(\lambda_{1,n}\frac{u_n}{|x|^{2a}},\lambda_{2,n}\frac{v_n}{|x|^{2a}}\right)}\to 0
		\end{align*}
	with 
	\begin{align*}  
		\lambda_{1,n}=-\frac{1}{\rho_1^2}J'\v[(u_n,0)]\quad\text{and}\quad\lambda_{2,n}=-\frac{1}{\rho_2^2}J'\v[(0,v_n)].
		\end{align*}
	Using inequality \eqref{d5}, we see that $J$ is coercive on $S\r$, hence $\v$ is bounded in $\y$ and $\lambda_{1,n},~\lambda_{2,n}$ are bounded in $\R$. Then there exists $\u\in\y$
	and $\lambda_1,~\lambda_2\in \R$ such that up to a subsequence:
	\begin{align*}  
		&\v\rightharpoonup\u\quad\text{in}~\y,\\
		&\v\to\u\quad\text{in}~{L^{p+q}_r(\R^N;|x|^{-b(p+q)})\times L^{p+q}_r(\R^N;|x|^{-b(p+q)})},\\
		&\v\to\u\quad\text{a.e. in }
		{\R^N}\\ &(\lambda_{1,n},\lambda_{2,n})\to(\lambda_{1},\lambda_{2})\quad\text{ in }
		{\R^2}.
	\end{align*}
	We claim that $u\ne0$ and $v\ne0$. Indeed, if $u=0$ or $v=0$, then $\b\v\to\b\u=0$ and
	\begin{align*}  
		\kappa\r=\lim_{n\to\infty}\left(\frac{1}{2}\left(\||x|^{-a}\na {u_n}\|_2^2+\||x|^{-a}\na {v_n}\|_2^2\right)-\b\v\right)\ge 0,
		\end{align*}
	a contradiction with $\kappa\r<0$. So we have $0<\||x|^{-a}u\|_2\le\rho_1$ and $0<\||x|^{-a}v\|_2\le \rho_2.$ Moreover, $\u$ is a solution of 
	$$\left\{
	\begin{aligned}
		-\text{div}(|x|^{-2a}\na u)+&\lambda_1 \frac{u}{|x|^{2a}}= p\frac{|v|^q|u|^{p-2}u}{|x|^{b(p+q)}},\\
		-\text{div}(|x|^{-2a}\na v)+&\lambda_2\frac{v}{|x|^{2a}}={q}\frac{|u|^p|v|^{q-2}v}{|x|^{b(p+q)}},\\
		u\ge 0,&\quad v\ge 0.
	\end{aligned}
	\right.
	$$ 
	We claim that $\lambda_1,~\lambda_2>0$.  Indeed, if $\lambda_1\le 0$ then
	\begin{align*}  -\text{div}(|x|^{-2a}\na u)=|\lambda_1| \frac{u}{|x|^{2a}}+ p\frac{|v|^q|u|^{p-2}u}{|x|^{b(p+q)}}\ge 0,\quad\text{in}~\R^N,
	\end{align*}
	and hence $u=0$ by Lemma~\ref{D2}, which is impossible. Analogously, we can prove $\lambda_2>0$. 
	Since $\v$ is a Palais--Smale sequence and $\u$ solves the system, for every $\varphi\in X$
	\begin{align}
		\label{d407}J'\v[(\varphi,0)]+\lambda_{1,n}\int_{\R^N}\frac{u_n\varphi}{|x|^{2a}}\dx&=o_n(1)\|\varphi\|_X,\\
		\label{d408}J'\u[(\varphi,0)]+\lambda_{1}\int_{\R^N}\frac{u\varphi}{|x|^{2a}}\dx&=0.
	\end{align}
	Taking $\varphi=u_n-u$ in \eqref{d407} and \eqref{d408} and subtracting,
	\begin{align*}  
		{o_n(1)}&{=\left[ J'\v-J'\u\right]\left[(u_n-u,0) \right]+\lambda_{1,n}\int_{\R^N}\frac{u_n(u_n-u)}{|x|^{2a}}\dx-\lambda_{1}\int_{\R^N}\frac{u(u_n-u)}{|x|^{2a}}\dx}\\
		&{=\||x|^{-a}\na(u_n-u)\|_2^2+\lambda_{1,n}\||x|^{-a}(u_n-u)\|_2^2+(\lambda_{1,n}-\lambda_1)\int_{\R^N}\frac{u(u_n-u)}{|x|^{2a}}\dx+o_n(1). }
	\end{align*}
Here we used
\begin{align*}
	\left[J'\v-J'\u\right]\left[(u_n-u,0)\right]=\||x|^{-a}\na(u_n-u)\|_2^2+o_n(1),\quad
		\int_{\R^N}\frac{u(u_n-u)}{|x|^{2a}}\dx=o_n(1).
\end{align*}
Since $\lambda_{1,n}\to\lambda_1>0$, we get $u_n\to u$ in $X$. Analogously we can prove $v_n\to v$ in $X$. Then $\u\in \mathcal{W}\r$, and $J\u=\kappa\r$. Clearly, $\u$ is a solution of {\eqref{d60}} and therefore a normalized ground state of {\eqref{d60}}. Moreover, $u,~v$ are positive by the maximum principle.
	
	If $p+q>(p+q)_c$ one can construct a Palais-Smale sequence $\v\in S\r$ for $J|_{S\r}$ at level $\kappa\r>0$ with $P\v\to0$ and $u_n^-,~v_n^-\to0$ a.e. in $\R^N$ as $n\to\infty.$ Then one can argue as in the case $p+q<(p+q)_c$; the only change is in the proof that $u\ne0$ and $v\ne0$: if $u=0$ or $v=0$ then $\b\v\to0$, and $P\v\to0$ forces $\||x|^{-a}\na u_n\|_2^2+\||x|^{-a}\na v_n\|_2^2\to0$, whence $\kappa\r=0$, contradicting Lemma~\ref{D106}~\textit{(ii)}.
	\qed
	\section{Asymptotic behavior as $\beta\ra 0^+$ and $\beta \ra \infty$}\label{SD5}
	In this section, we study the asymptotic behavior of the ground states $(u_\beta,v_\beta)$ of Theorems~\ref{DD2} and \ref{DD6} as $\beta\to 0^+$ or $\beta\to+\infty$. For simplicity, we write $A_\beta\sim B_\beta,~A_\beta\lesssim B_\beta$ and $A_\beta\gtrsim B_\beta$ if there exist $C_1,~C_2>0$ such that, respectively, $C_1B_\beta\le A_\beta\le C_2B_\beta,~A_\beta \le C_1 B_\beta$ and $A_\beta\ge C_1 B_\beta$ as $\beta\to 0^+$ (or $\beta\to \infty$).We first study the limit
	$
	\beta\to0^+.
	$
	
	The following result gives the first-order information on the
	energy scale and the Lagrange multipliers.
	\begin{lemma}\label{D51}
		For $p+q<(p+q)_c,$
		\begin{align*}  \lambda_{1,\beta}+\lambda_{2,\beta}\sim \||x|^{-a}\na u_\beta\|_2^2+\||x|^{-a}\na v_\beta\|_2^2\sim \beta^{\frac{2}{2-\pq}},\end{align*}
		as $\beta\to 0^+.$
	\end{lemma}
	\begin{proof}
		Since $(u_\beta, v_\beta)\in \mathcal{N}_\beta^+\r$, we have
		\begin{equation}\label{d105}
			\||x|^{-a}\na u_\beta\|_2^2+\||x|^{-a}\na v_\beta\|_2^2=\| |x|^{-b}u_\beta\|_\t^\t+\||x|^{-b}v_\beta\|_\t^\t+\beta\pq\b(u_\beta, v_\beta),
		\end{equation}
		and
		\begin{equation}\label{d107}
			2\left(\||x|^{-a}\na u_\beta\|_2^2+\||x|^{-a}\na v_\beta\|_2^2\right)\ge \t\left(\| |x|^{-b}u_\beta\|_\t^\t+\||x|^{-b}v_\beta\|_\t^\t\right)+\beta(p+q)^2\delta_{p+q}^2\b(u_\beta,v_\beta).
		\end{equation}
		It follows from inequality \eqref{d5} that
		\begin{align*}  \||x|^{-a}\na u_\beta\|_2^2+\||x|^{-a}\na v_\beta\|_2^2\lesssim \beta\b(u_\beta,v_\beta)\lesssim\beta\left(\||x|^{-a}\na u_\beta\|_2^2+\||x|^{-a}\na v_\beta\|_2^2\right)^{\frac{\pq}{2}},\end{align*}
		which implies
		\begin{equation}\label{d108}
			\||x|^{-a}\na u_\beta\|_2^2+\||x|^{-a}\na v_\beta\|_2^2\lesssim \beta^{\frac{2}{2-\pq}}
		\end{equation}
		and, by \eqref{d105}
		\begin{equation}\label{d106}
			\beta \b(u_\beta,v_\beta)\lesssim \beta^{\frac{2}{2-\pq}}.
		\end{equation}
		For $\u\in S\r$ and $\beta>0$ small there exists by Lemma~\ref{D3} a unique $t_\beta:=t^\u_\beta$ such that $t_\beta\star\u\in \mathcal{N}_\beta^+\r$, i.e
		\begin{align*}
        e^{2t_\beta}\left(\||x|^{-a}\na u\|_2^2+\||x|^{-a}\na v\|_2^2\right)=& e^{\t t_\beta}\left(\| |x|^{-b}u\|_\t^\t+\||x|^{-b}v\|_\t^\t\right)\\
        &+\beta\pq e^{\pq t_\beta}\b(u, v),
        \end{align*}
		and
		\begin{align*}  2e^{2t_\beta}\left(\||x|^{-a}\na u\|_2^2+\||x|^{-a}\na v\|_2^2\right)>&\t e^{\t t_\beta}\left(\| |x|^{-b}u\|_\t^\t+\||x|^{-b}v\|_\t^\t\right)\\& +\beta(p+q)^2\delta_{p+q}^2 e^{\pq t_\beta}\b(u, v).\end{align*}
		Therefore
		\begin{align*}  (\t-2)e^{2t_\beta}\left(\||x|^{-a}\na u\|_2^2+\||x|^{-a}\na v\|_2^2\right)<\beta (\t-\pq)\pq e^{\pq t_\beta}\b(u, v),\end{align*}
		
		which implies $e^{t_\beta}\to 0$ as $\beta\to 0$ because $\pq <2$. It follows that
		\begin{align*}  e^{2t_\beta}\sim \beta e^{\pq t_\beta},\quad \text{as}~\beta\to 0,\end{align*}
		which implies $e^{t_\beta}\sim\beta^{\frac{1}{2-\pq}}$ as $\beta\to0$, hence
		\begin{align*}  
			E_\beta(t_\beta\star\u)&=\frac{\pq-2}{2\pq}e^{2t_\beta}\left(\||x|^{-a}\na u\|_2^2+ \||x|^{-a}\na v\|_2^2\right) \\& +\frac{\t-\pq}{\t \pq}e^{\t t_\beta}\Big(\| |x|^{-b}u\|_\t^\t\qquad+\||x|^{-b}v\|_\t^\t\Big)\\
			&\sim -\beta^{\frac{2}{2-\pq}}.
		\end{align*}
		Using $E_\beta(t_\beta\star\u)\ge m_\beta\r$ and $m_\beta\r\gtrsim -\beta\b(u_\beta,v_\beta)$, next we deduce
		\begin{align*}  \beta\b(u_\beta,v_\beta)\gtrsim\beta^{\frac{2}{2-\pq}}, \end{align*}
		which together with \eqref{d106}, implies
		\begin{align*}  \beta\b(u_\beta,v_\beta)\sim \beta^{\frac{2}{2-\pq}}.\end{align*}
		Testing the two equations of \eqref{d1} with $(u_\beta,0)$ and $(0,v_\beta)$,
		\begin{align*}  \lambda_{1,\beta}+\lambda_{2,\beta}\sim\beta\b(u_\beta,v_\beta)\sim \beta^{\frac{2}{2-\pq}},\end{align*}
		and by \eqref{d107} and \eqref{d108} we get
		\begin{align*}  \||x|^{-a}\na {u_\beta}\|_2^2+\||x|^{-a}\na {v_\beta}\|_2^2\sim \beta\b(u_\beta,v_\beta)\sim \beta^{\frac{2}{2-\pq}}\end{align*}
		which completes the proof.\qed
	\end{proof}
	By Lemma~\ref{D105} we see that for any $\u\in S\r$ 
	\begin{equation}\label{d110}
		\b\u\le C_0\left(\||x|^{-a}\na u\|_2^2+\||x|^{-a}\na v\|_2^2\right)^{\frac{\pq}{2}}
	\end{equation}
	with \begin{align*}  C_0:=\frac{1}{\pq}\left(\frac{\pq-2}{2\pq\kappa\r}\right)^{\frac{\pq-2}{2}}>0.\end{align*}
	\begin{lemma}\label{D52}
		If $p+q<(p+q)_c$ then
		\begin{align*}  \text{dist}_{\mathcal{X}}((-t_\beta)\star(u_\beta,v_\beta),{\mathcal{Z}}\r)\to0\quad \text{as}~\beta\to 0^+,\end{align*}
		where $t_\beta:=t_\beta^{(u_0,v_0)}$ is from Lemma~\ref{D3}. Moreover we have $e^{t_\beta}\sim\beta^{\frac{1}{{2-\pq}}}$ as $\beta\to 0^+.$
	\end{lemma}
	\begin{proof}
		Let $(u_0,v_0)\in {\mathcal{Z}}\r.$ Then ${\widetilde{J}}(u_0,v_0)=\kappa\r$ which implies
		\begin{equation}\label{d111}
			\b(u_0,v_0)=C_0\left(\||x|^{-a}\na u_0\|_2^2+\||x|^{-a}\na v_0\|_2^2\right)^{\frac{\pq}{2}}
		\end{equation}
		Since $t_\beta{\star(u_0,v_0)}\in \mathcal{N}_\beta^+\r$ we have
		\begin{align}\label{d127}
			\notag e^{2t_\beta}\left(\||x|^{-a}\na u_0\|_2^2+\||x|^{-a}\na v_0\|_2^2\right)&=e^{\t t_\beta}\left(\||x|^{-b}u_0\|_\t^\t+\||x|^{-b}v_0\|_\t^\t\right)\\&\quad+e^{\pq t_\beta}\pq \beta\b(u_0,v_0),
		\end{align}
		and
		\begin{align}\label{d128}
			\notag 2e^{2t_\beta}\left(\||x|^{-a}\na u_0\|_2^2+\||x|^{-a}\na v_0\|_2^2\right)&>\t e^{\t t_\beta}\left(\||x|^{-b}u_0\|_\t^\t+\||x|^{-b}v_0\|_\t^\t\right)\\&\quad+e^{\pq t_\beta}(p+q)^2\delta_{p+q}^2 \beta\b(u_0,v_0).
		\end{align}
		Using \eqref{d110} we have
	\begin{equation}
	\begin{aligned}\label{d116}
			 & (\t-2)\left(\||x|^{-a}\na (t_\beta \star u_0)\|_2^2+\||x|^{-a}\na (t_\beta \star v_0)\|_2^2\right)\\
             & \qquad  <(\t-\pq) \pq\beta\b(t_\beta\star(u_0,v_0))\\
			& \qquad <(\t-\pq) \pq\beta C_0 \left(\||x|^{-a}\na (t_\beta \star u_0)\|_2^2+\||x|^{-a}\na (t_\beta \star v_0)\|_2^2\right)^{{\frac{\pq}{2}}}
		\end{aligned}
        \end{equation}	
		which yields
		\begin{equation}\label{d112}
			\||x|^{-a}\na (t_\beta \star u_0)\|_2^2+\||x|^{-a}\na (t_\beta \star v_0)\|_2^2<\left(\frac{\t-\pq}{\t-2}\pq \beta C_0\right)^{\frac{2}{2-\pq}}.
		\end{equation}
		Since $(u_\beta,v_\beta)\in \mathcal{N}_\beta^+\r$ we observe that
		\begin{equation}\label{d113}
			\||x|^{-a}\na u_\beta\|_2^2+\||x|^{-a}\na v_\beta\|_2^2<\left(\frac{\t-\pq}{\t-2}\pq \beta C_0\right)^{\frac{2}{2-\pq}}.
		\end{equation}
		Now using $t_\beta\star(u_0,v_0)$ as a test function of $m_\beta\r$ and by \eqref{d111} we get
		\begin{align}\label{d114}
			\notag m_\beta\r &\le E_\beta (t_\beta\star (u_0,v_0))\\\notag&=\frac{d}{N}\left(\||x|^{-a}\na (t_\beta \star u_0)\|_2^2+\||x|^{-a}\na (t_\beta \star v_0)\|_2^2\right)\\
			&\qquad-\beta \frac{\t-\pq}{\t}C_0\left(\||x|^{-a}\na (t_\beta \star u_0)\|_2^2+\||x|^{-a}\na (t_\beta \star v_0)\|_2^2\right)^{\frac{\pq}{2}}.
		\end{align}
		With the help of inequality  \eqref{d110},
		\begin{align}\label{d115}
			\notag m_\beta\r&=E_\beta(u_\beta,v_\beta)\\\notag&\ge\frac{d}{N}\left(\||x|^{-a}\na u_\beta\|_2^2+\||x|^{-a}\na v_\beta\|_2^2\right)\\&\qquad-\beta\frac{\t-\pq}{\t}C_0\left(\||x|^{-a}\na u_\beta\|_2^2+\||x|^{-a}\na v_\beta\|_2^2\right)^{\frac{\pq}{2}}.
		\end{align}
		A direct calculation shows that the function
		\begin{align*}  h(t):=\frac{d}{N}t-\beta\frac{\t-\pq}{\t}C_0 t^{\frac{\pq}{2}},\end{align*}
		is  strictly decreasing in $(0,t_0)$, where
		\begin{align*}  t_0:=\left(\frac{\t-\pq}{\t-2}\pq \beta C_0\right)^{\frac{2}{2-\pq}}.\end{align*}
		Thus \eqref{d112}-\eqref{d115} yield
		\begin{equation}\label{d132}
			\||x|^{-a}\na u_\beta\|_2^2+\||x|^{-a}\na v_\beta\|_2^2\ge \||x|^{-a}\na (t_\beta \star u_0)\|_2^2+\||x|^{-a}\na (t_\beta \star v_0)\|_2^2. 
		\end{equation}
		Moreover, similar to \eqref{d116}-\eqref{d115}, we have
		\begin{equation}\label{d121}
			\b(t_\beta\star(u_0,v_0))<\left(\frac{\t-\pq}{\t-2}\pq \beta C_0^{\frac{2}{\pq}}\right)^{\frac{\pq}{2-\pq}}
		\end{equation} 
		and 
		\begin{equation}\label{d122}
			\b(u_\beta,v_\beta)<\left(\frac{\t-\pq}{\t-2}\pq \beta C_0^{\frac{2}{\pq}}\right)^{\frac{\pq}{2-\pq}}.
		\end{equation}
		Now using $t_\beta\star(u_0,v_0)$ again as a test function of $m_\beta\r$ and by \eqref{d111}
		\begin{align}\label{d123}
			\notag  m_\beta\r\le E_\beta(t_\beta\star(u_0,v_0))=\frac{d}{N}C_0^{-\frac{2}{\pq}}\b (t_\beta\star(u_0,v_0))^{\frac{2}{\pq}}\\-\beta\frac{\t-\pq}{\t}\b(t_\beta\star(u_0,v_0))
		\end{align}
		Applying inequality \eqref{d110}, we infer
		\begin{equation}\label{d124}
			m_\beta\r=E_\beta(u_\beta,v_\beta)\ge\frac{d}{N}C_0^{\frac{-2}{\pq}}\b(u_\beta,v_\beta)^{\frac{2}{\pq}}
			-\beta\frac{\t-\pq}{\t}\b(u_\beta,v_\beta).
		\end{equation}
		The function 
		\begin{align*}  h(t)=\frac{d}{N}C_0^{\frac{-2}{\pq}}t^{
				\frac{2}{\pq}}-\beta\frac{\t-\pq}{\t}t\end{align*}
		is strictly decreasing in $(0,t_0)$, where
		\begin{align*}  t_0:=\left(\frac{\t-\pq}{\t-2}\pq \beta C_0^{\frac{2}{\pq}}\right)^{\frac{\pq}{2-\pq}}.\end{align*}
		Therefore,  by \eqref{d121}-\eqref{d124}, we have
		\begin{equation}\label{d213}
			\b(u_\beta, v_\beta)\ge \b(t_\beta\star(u_0,v_0)).
		\end{equation}
		Just like in Lemma~\ref{D51}, from \eqref{d127} and \eqref{d128}, it follows
		\begin{align*}  e^{t_\beta}\sim \beta^{\frac{1}{2-\pq}}\quad\text{as}~\beta\to0.\end{align*}
		Then using \eqref{d127}, we obtain
		\begin{equation}\label{d212}
			e^{t_\beta}=(1+o_\beta(1))\left(\frac{\beta\pq \b(u_0,v_0)}{\||x|^{-a}\na u_0\|_2^2+\||x|^{-a}\na v_0\|_2^2}\right)^{\frac{1}{2-\pq}}=(1+o_\beta(1))\beta^{\frac{1}{2-\pq}}.
		\end{equation}
		Set $(\tilde u_\beta,\tilde v_\beta):=(-t_\beta)\star(u_\beta,v_\beta)$. Since $(u_\beta,v_\beta)$ satisfies \eqref{d1}, $(\tilde{u}_\beta,\tilde{v}_\beta)$ satisfies the following equation
        \begin{small}
		\begin{equation}\label{d131}
			\begin{cases}
				-\text{div}(|x|^{-2a}\nabla \tilde{u}_\beta)
				+ \lambda_{1,\beta}e^{-2t_\beta} \dfrac{\tilde{u}_\beta}{|x|^{2a}}
				= \beta pe^{(\pq-2)t_\beta}\dfrac{|\tilde{v}_\beta|^{q}|\tilde{u}_\beta|^{p-2}\tilde{u}_\beta}{|x|^{b(p+q)}}
				+ e^{(\t-2)t_\beta}\dfrac{|\tilde{u}_\beta|^{2^{\sharp}-2}\tilde{u}_\beta}{|x|^{b2^{\sharp}}},~\text{in}~\R^N \\[0.3cm]
				-\text{div}(|x|^{-2a}\nabla \tilde{v}_\beta)
				+ \lambda_{2,\beta}e^{-2t_\beta} \dfrac{\tilde{v}_\beta}{|x|^{2a}}
				= \beta q e^{((p+q)\delta_{p+q}-2)t_\beta}\dfrac{|\tilde{u}_\beta|^{p}|\tilde{v}_\beta|^{q-2}\tilde{v}_\beta}{|x|^{b(p+q)}}
				+ e^{(\t-2)t_\beta}\dfrac{|\tilde{v}_\beta|^{2^{\sharp}-2}\tilde{v}_\beta}{|x|^{b2^{\sharp}}}, ~\text{in}~\R^N\\[0.3cm]
				\displaystyle
				\int_{\R^N}\frac{|\tilde{u}_\beta|^2}{|x|^{2a}}\dx=\rho_1^2, \qquad
				\int_{\R^N}\frac{|\tilde{v}_\beta|^2}{|x|^{2a}}\dx=\rho_2^2.
			\end{cases}
		\end{equation}
        \end{small}
		Taking into account Lemma~\ref{D51}, we get
		\begin{align*}  \||x|^{-a}\na \tilde{u}_\beta\|_2^2+\||x|^{-a}\na \tilde{v}_\beta\|_2^2=e^{-2t_\beta}\left(\||x|^{-a}\na u_\beta\|_2^2+\||x|^{-a}\na v_\beta\|_2^2\right)\sim 1.\end{align*}
		Therefore $\{(\tilde{u}_\beta,\tilde{v}_\beta)\}$ is bounded in $\y$. Moreover, up to a subsequence 
		\begin{align*}  
			(\tilde{u}_\beta,\tilde{v}_\beta)&\rightharpoonup(\tilde{u},\tilde{v})\quad\text{in}~\y,\\
			(\tilde{u}_\beta,\tilde{v}_\beta)&\to(\tilde{u},\tilde{v})\quad\text{in}~L^{p+q}_r(\R^N;|x|^{-b(p+q)})\times L^{p+q}_r(\R^N;|x|^{-b(p+q)}),\\
			(\tilde{u}_\beta,\tilde{v}_\beta)&\to(\tilde{u},\tilde{v})\quad\text{a.e in}~\R^N.
		\end{align*}
		Again using Lemma~\ref{D51}, we observe that $\{\lambda_{{1,\beta}}e^{-2t_\beta}\}$ and $\{\lambda_{{2,\beta}}e^{-2t_\beta}\}$ are bounded, hence
		\begin{align*}  \lambda_{1,\beta}e^{-2t_\beta}\to \lambda_1,\quad\text{and}\quad\lambda_{2,\beta}e^{-2t_\beta}\to\lambda_2,\quad\text{as}\beta\to0\end{align*}
		up to a subsequence. On the other hand by \eqref{d212}
		\begin{align*}  e^{(\t-2)t_\beta}\to0,\quad\text{and}\quad\beta e^{(\pq-2)t_\beta}\to 1,\quad\text{as}~\beta\to0.\end{align*}
		Thus, we know that
		\begin{equation*}
			\begin{cases}
				-\text{div}(|x|^{-2a}\nabla \tilde{u})
				+ \lambda_{1} \dfrac{\tilde{u}}{|x|^{2a}}
				= p\dfrac{|\tilde{v}|^{q}|\tilde{u}|^{p-2}\tilde{u}}{|x|^{b(p+q)}}, \\[0.3cm]
				-\text{div}(|x|^{-2a}\nabla \tilde{v})
				+ \lambda_{2}\dfrac{\tilde{v}}{|x|^{2a}}
				= q \dfrac{|\tilde{u}|^{p}|\tilde{v}|^{q-2}\tilde{v}}{|x|^{b(p+q)}}
			\end{cases}
		\end{equation*}
		Now \eqref{d213} implies $\b(\tilde{u},\tilde{v})\ge \b(u_0,v_0)$, which gives $\tilde{u}\ne0,~\tilde{v}\ne 0.$ By the maximum principle and  Lemma~\ref{D2}, we obtain $\lambda_1>0$ and $\lambda_2>0.$ Then testing \eqref{d131} with $(\tilde u_\beta-\tilde u,0)$ and $(0,\tilde v_\beta-\tilde v)$ and arguing as in the proof of Theorem~\ref{DD4}, the critical terms being controlled by $e^{(\t-2)t_\beta}\to0$, we conclude $(\tilde{u}_\beta,\tilde{v}_\beta)\to(\tilde{u},\tilde{v})$ in $\y$ as $\beta\to0.$ It follows that $\||x|^{-a}\tilde{u}\|_2^2=\rho_1^2,~\||x|^{-a}\tilde{v}\|_2^2=\rho_2^2$ and by \eqref{d132}
		\begin{align*}  
			\kappa\r\le J(\tilde{u},\tilde{v})&=\frac{\pq-2}{2\pq}\left(\||x|^{-a}\na \tilde{u}\|_2^2+\||x|^{-a}\na\tilde{v}\|_2^2\right)\\&\le \frac{\pq-2}{2\pq}\left(\||x|^{-a}\na u_0\|_2^2+\||x|^{-a}\na v_0\|_2^2\right)\\
			& =J(u_0,v_0)=\kappa\r.
		\end{align*}
		Finally, we have $(\tilde{u},\tilde{v})\in\mathcal{Z}\r$. Therefore, the lemma holds.\qed
	\end{proof}
	Now, for the case $p+q\ge\p$, we have the following lemma.
	\begin{lemma}\label{D53}
		For $p+q\ge \p$, we have
		\begin{align*}  \text{dist}_{\mathcal{D}_a^{1,2}\times\mathcal{D}_a^{1,2}}((u_\beta,v_\beta), \mathcal{U}\times \{0\})\to 0\quad \text{as}~\beta\to0^+,\end{align*}
		or  \begin{align*}  \text{dist}_{\mathcal{D}_a^{1,2}\times\mathcal{D}_a^{1,2}}((u_\beta,v_\beta),  \{0\}\times\mathcal{U})\to 0\quad \text{as}~\beta\to0^+.\end{align*}
		here $\mathcal{U}=\{U_\var:\var>0\}.$
	\end{lemma}
	\begin{proof}
		First observe that $\{(u_\beta,v_\beta)\}$ is bounded in $\y$ because $(u_\beta,v_\beta)\in \mathcal{N}_\beta\r$ satisfies $E_\beta(u_\beta,v_\beta)=m_\beta\r$. Let $\||x|^{-a}\na u_\beta\|_2\to \ell_1$ and $\||x|^{-a}\na v_\beta\|_2\to \ell_2$ with $\ell_1,~\ell_2\ge 0$. If possible suppose $\ell_1=\ell_2=0$ then
		\begin{align*}  0=\lim_{\beta\to0}E_\beta(u_\beta,v_\beta)=\lim_{\beta\to0}m_\beta\r>0,\end{align*}
		a contradiction. Definition of $\mathcal{S}(a,b)$ and Pohozaev identity $P_\beta(u_\beta,v_\beta)= 0$ imply
		\begin{align*}  \ell_1^2+\ell_2^2\le \mathcal{S}(a,b)^{-\t/2}(\ell_1^\t+\ell_2^\t).\end{align*}
		This inequality forces $\ell_1^2 + \ell_2^2 \ge \mathcal{S}(a,b)^{N/2d}$ and equality holds only if either $\ell_1=0$ or $\ell_2=0$. Subsequently we have
		\begin{align*}  \frac{d}{N}\mathcal{S}(a,b)^{N/2d}\ge\lim_{\beta\to0^+}E_\beta(u_\beta,v_\beta)=\frac{d}{N}(\ell_1^2+\ell_2^2).\end{align*}
		Therefore, when $\ell_1=0$, then $\ell_2^2=\mathcal{S}(a,b)^{N/2d}$ and we have
		\begin{align*}  \||x|^{-a}\na u_\beta\|_2^2\to 0,\quad\||x|^{-a}\na v_\beta\|_2^2\to\mathcal{S}(a,b)^{N/2d},\quad \||x|^{-b}v_\beta\|_\t^\t\to \mathcal{S}(a,b)^{N/2d}\quad \text{as}~\beta\to0^+,\end{align*}
		hence $v_\beta$  is a minimizing sequence for \eqref{d191}. Now \cite[Theorem 1.41]{willem2012minimax} implies that there exists $l_\beta>0$ and, along a subsequence, some $\var_*>0$ such that
		\begin{align*}  \left(u_\beta, l_\beta^{\frac{N-2-2a}{2}}v_\beta(l_\beta\,\cdot)\right)\to (0, U_{\var_*})\quad\text{ strongly in }~\mathcal{D}_a^{1,2}\times\mathcal{D}_a^{1,2}~\text{as}~\beta\to 0^+,\end{align*}
		up to a subsequence. The case when $\ell_1^2=\mathcal{S}(a,b)^{N/2d},~\ell_2=0$ can be treated in similar manner. \qed
	\end{proof}
    \noi
    \textbf{Proof of Theorem~\ref{DD11}} The proof follows immediately from Lemmas~\ref{D51},~\ref{D52}, and~\ref{D53}. \qed
    
	\noi
We now prove Theorem~\ref{DD10}. Throughout, $\pq>2$.
For $\var\ge0$ we set
\begin{align}\label{d420}
J_\var\u:=\tfrac12 \||x|^{-a}\na u\|_2^2+\||x|^{-a}\na v\|_2^2-\b\u-\tfrac{\var}{\t}(\||x|^{-b}u\|_\t^\t+\||x|^{-b}v\|_\t^\t) ,
\end{align}
\begin{equation*}
    P^\var\u:= \||x|^{-a}\na u\|_2^2+\||x|^{-a}\na v\|_2^2-\pq\b\u-\var(\||x|^{-b}u\|_\t^\t+\||x|^{-b}v\|_\t^\t) ,
\end{equation*}
so that $J_0=J$ and $\mathcal W\r=\{\u\in S\r:P^0\u=0\}$. Finally we put
\begin{align*}
s_\beta:=\frac{\ln\beta}{\pq-2},\qquad
(\hat u_\beta,\hat v_\beta):=s_\beta\star(u_\beta,v_\beta)\in S\r,\qquad
\var_\beta:=\beta^{-\frac{\t-2}{\pq-2}}\to0 .
\end{align*}

\begin{lemma}\label{D306}
For every $\beta>\beta_0$,
\begin{align}\label{d421}
\beta^{\frac{2}{\pq-2}}E_\beta(u_\beta,v_\beta)=J_{\var_\beta}(\hat u_\beta,\hat v_\beta),
\qquad
\beta^{\frac{2}{\pq-2}}P_\beta(u_\beta,v_\beta)=P^{\var_\beta}(\hat u_\beta,\hat v_\beta)=0 .
\end{align}
Moreover $(\hat u_\beta,\hat v_\beta)$ is a critical point of $J_{\var_\beta}|_{S\r}$ with
multipliers $\hat\lambda_{i,\beta}:=\beta^{\frac{2}{\pq-2}}\lambda_{i,\beta}>0$, the map
$t\mapsto J_{\var_\beta}(t\star(\hat u_\beta,\hat v_\beta))$ attains its maximum at $t=0$, and
\begin{align}\label{d422}
\beta^{\frac{2}{\pq-2}}m_\beta\r=\inf_{\u\in S\r}\max_{t\in\R}J_{\var_\beta}(t\star\u).
\end{align}
\end{lemma}
\begin{proof}
Since $e^{s_\beta}=\beta^{\frac1{\pq-2}}$ gives
\begin{align*}
e^{-2s_\beta}=\beta e^{-\pq s_\beta}=\beta^{-\frac{2}{\pq-2}},\qquad
e^{-\t s_\beta}=\var_\beta\,e^{-2s_\beta},
\end{align*}
multiplying $E_\beta(u_\beta,v_\beta)=E_\beta((-s_\beta)\star(\hat u_\beta,\hat v_\beta))$ by
$e^{2s_\beta}$ gives the first identity in \eqref{d421}, and the same computation for $P_\beta$
gives the second. Thus $J_{\var_\beta}=\beta^{\frac{2}{\pq-2}}E_\beta\circ((-s_\beta)\star\,\cdot\,)$
on $S\r$. As $(-s_\beta)\star\,\cdot\,$ is a diffeomorphism of $S\r$ which leaves the two weighted
masses unchanged, constrained critical points correspond to each other and the multipliers are
multiplied by $\beta^{\frac{2}{\pq-2}}$. Since $s_\beta\star(t\star\u)=t\star(s_\beta\star\u)$,
the same identity gives $\beta^{\frac{2}{\pq-2}}\max_tE_\beta(t\star\u)=\max_tJ_{\var_\beta}(t\star(s_\beta\star\u))$;
taking the infimum over $\u\in S\r$ and using Lemma~\ref{D7} yields \eqref{d422}, and taking
$\u=(u_\beta,v_\beta)$ together with Lemma~\ref{D6}(iii) shows that $t=0$ maximizes
$t\mapsto J_{\var_\beta}(t\star(\hat u_\beta,\hat v_\beta))$.\qed
\end{proof}

\begin{lemma}\label{D307}
There exist $0<c_1\le c_2$ and $\beta_1>\beta_0$ such that, for all $\beta\ge\beta_1$,
\begin{align*}
c_1\le \||x|^{-a}\na \hat u\|_2^2+\||x|^{-a}\na\hat v\|_2^2\le c_2,\quad
c_1\le\b(\hat u_\beta,\hat v_\beta),\\
\||x|^{-b}\hat u\|_\t^\t+\||x|^{-b}\hat v\|_\t^\t\le c_2,\qquad
0<\hat\lambda_{i,\beta}\le c_2 .
\end{align*}
\end{lemma}
\begin{proof}
 Since $J_{\var_{\beta}}\le J_0$ and $\max_tJ_0(t\star\u)=\widetilde J\u$, \eqref{d422}
and Lemma~\ref{D105} give $\beta^{\frac{2}{\pq-2}}m_\beta\r\le\inf_{S\r}\widetilde J=\kappa\r$.
By \eqref{d420} and $P^{\var_\beta}(\hat u_\beta,\hat v_\beta)=0$,
\begin{align*}
\frac{\pq-2}{2\pq} \||x|^{-a}\na \hat u\|_2^2+\||x|^{-a}\na\hat v\|_2^2+\var\frac{\t-\pq}{\pq\t}\||x|^{-b}\hat u\|_\t^\t+\||x|^{-b}\hat v\|_\t^\t\\
=J_{\var_\beta}(\hat u_\beta,\hat v_\beta)-\frac1\pq P^{\var_\beta}(\hat u_\beta,\hat v_\beta)
=\beta^{\frac{2}{\pq-2}}m_\beta\r\le\kappa\r ,
\end{align*}
and both terms on the left are nonnegative because $2<\pq<\t$. This bounds $ \||x|^{-a}\na \hat u\|_2^2+\||x|^{-a}\na\hat v\|_2^2$, and the
Caffarelli--Kohn--Nirenberg inequality then gives \begin{align*}
    \||x|^{-b}\hat u\|_\t^\t+\||x|^{-b}\hat v\|_\t^\t\le c_2\le\sr^{-\t/2}\left(\||x|^{-a}\na \hat u\|_2^2+\||x|^{-a}\na\hat v\|_2^2\right)^{\t/2}\le c_2.\end{align*}
For the lower bound, $P^{\var_\beta}(\hat u, \hat v)=0$ and \eqref{d5} on $S\r$ give
\begin{align*}
    \||x|^{-a}\na \hat u\|_2^2+\||x|^{-a}\na\hat v\|_2^2&=\pq{\mathcal B}(\hat u, \hat v)+\var_\beta\left( \||x|^{-b}\hat u\|_\t^\t+\||x|^{-b}\hat v\|_\t^\t\right)\\& \le\pq C\left(\||x|^{-a}\na \hat u\|_2^2+\||x|^{-a}\na\hat v\|_2^2\right)^{\pq/2}\\&\quad+\var_{\beta}\sr^{-\t/2}\left(\||x|^{-a}\na \hat u\|_2^2+\||x|^{-a}\na\hat v\|_2^2\right)^{\t/2}\end{align*}
with $C=C_{a,b}^{p+q}(\rho_1^2+\rho_2^2)^{\frac{(p+q)(1-\delta_{p+q})}{2}}$. Dividing by $(\||x|^{-a}\na \hat u\|_2^2+\||x|^{-a}\na\hat v\|_2^2)>0$,
\begin{align*}
1\le\pq C\left(\||x|^{-a}\na \hat u\|_2^2+\||x|^{-a}\na\hat v\|_2^2\right)^{\frac{\pq-2}{2}}\\+\var\sr^{-\t/2}\left(\||x|^{-a}\na \hat u\|_2^2+\||x|^{-a}\na\hat v\|_2^2\right)^{\frac{\t-2}{2}} .
\end{align*}
As $\left(\||x|^{-a}\na \hat u\|_2^2+\||x|^{-a}\na\hat v\|_2^2\right)\le c_2$ and $\var\to0$, the second term tends to $0$, so $$\left(\||x|^{-a}\na \hat u\|_2^2+\||x|^{-a}\na\hat v\|_2^2\right)\ge c_1>0$$ for
$\beta$ large, and then $$\pq\hat{\mathcal B}=\||x|^{-a}\na \hat u\|_2^2+\||x|^{-a}\na\hat v\|_2^2-\var\left( \||x|^{-b}\hat u\|_\t^\t+\||x|^{-b}\hat v\|_\t^\t\right)\ge c_1/2.$$ Finally, testing the
equation for $\hat u_\beta$ with $\hat u_\beta$,
\begin{align*}
\hat\lambda_{1,\beta}\rho_1^2
=p\,\hat{\mathcal B}+\var\||x|^{-b}\hat u_\beta\|_\t^\t-\||x|^{-a}\na\hat u_\beta\|_2^2\le pc_2+c_2 ,
\end{align*}
and $\hat\lambda_{1,\beta}>0$ by Theorem~\ref{DD6}; similarly for $\hat\lambda_{2,\beta}$.\qed
\end{proof}

\begin{lemma}\label{D308}
$\displaystyle\lim_{\beta\to+\infty}\beta^{\frac{2}{\pq-2}}m_\beta\r=\kappa\r$.
\end{lemma}
\begin{proof}
The inequality $\limsup\le\kappa\r$ was obtained in the proof of Lemma~\ref{D307}. For the
converse, let $\hat t_\beta:=t^{(\hat u_\beta,\hat v_\beta)}$, so that
$e^{(\pq-2)\hat t_\beta}=\left(\||x|^{-a}\na \hat u\|_2^2+\||x|^{-a}\na\hat v\|_2^2\right)/(\pq\hat{\mathcal B})$. By Lemma~\ref{D307}, $\hat t_\beta$ is
bounded from above. Using Lemma~\ref{D105}, then \eqref{d420}, and then the last assertion of
Lemma~\ref{D306},
\begin{align*}
\\\kappa\r\le\widetilde J(\hat u_\beta,\hat v_\beta)
& =J_{\var_\beta}\big(\hat t_\beta\star(\hat u_\beta,\hat v_\beta)\big)
+\frac{\var_\beta}{\t}e^{\t\hat t_\beta}\left( \||x|^{-b}\hat u\|_\t^\t+\||x|^{-b}\hat v\|_\t^\t\right)\\
& 
\le J_{\var_\beta}(\hat u_\beta,\hat v_\beta)+O(\var_\beta)\\
& 
=\beta^{\frac{2}{\pq-2}}m_\beta\r+O(\var_\beta).
\end{align*}
Letting $\beta\to+\infty$ finishes the proof.\qed
\end{proof}
\noi\textbf{Proof of Theorem~\ref{DD10}} Let $\beta_n\to+\infty$ and write
$(\hat u_n,\hat v_n):=(\hat u_{\beta_n},\hat v_{\beta_n})$, $\hat\lambda_{i,n}$, $\var_n$. By
Lemma~\ref{D307} the sequence is bounded in $\y$, so, up to a subsequence,
$(\hat u_n,\hat v_n)\rightharpoonup(\hat u,\hat v)$ in $\y$, strongly in
$L^{p+q}_r(\R^N;|x|^{-b(p+q)})\times L^{p+q}_r(\R^N;|x|^{-b(p+q)})$ by Proposition~\ref{D303},
a.e.\ in $\R^N$, and $\hat\lambda_{i,n}\to\hat\lambda_i\ge0$. Passing to the limit in the
equations of Lemma~\ref{D306} we find that $(\hat u,\hat v)$ solves \eqref{d60} with multipliers
$(\hat\lambda_1,\hat\lambda_2)$; the critical term disappears because
$\var_n\||x|^{-b}\hat u_n\|_\t^{\t-1}\le C\var_n\to0$. By Lemma~\ref{D307},
$\b(\hat u,\hat v)=\lim\b(\hat u_n,\hat v_n)\ge c_1>0$, so $\hat u\ne0$ and $\hat v\ne0$, and
Lemma~\ref{D2} gives $\hat\lambda_1,\hat\lambda_2>0$ exactly as in the proof of Theorem
\ref{DD4}. The same testing argument as there, the critical terms being $O(\var_n)$, yields
$(\hat u_n,\hat v_n)\to(\hat u,\hat v)$ in $\y$. Consequently $(\hat u,\hat v)\in S\r$,
$P^0(\hat u,\hat v)=\lim P^{\var_n}(\hat u_n,\hat v_n)=0$, so $(\hat u,\hat v)\in\mathcal W\r$,
and $J(\hat u,\hat v)=\lim J_{\var_n}(\hat u_n,\hat v_n)=\kappa\r$ by Lemma~\ref{D308}. Hence
$(\hat u,\hat v)\in\mathcal Z\r$. Since every sequence $\beta_n\to+\infty$ has a subsequence
along which $s_{\beta_n}\star(u_{\beta_n},v_{\beta_n})$ converges in $\y$ to a point of
$\mathcal Z\r$, the distance tends to $0$.\qed

    \renewcommand{\thesection}{A}
	\section{Appendix: A Liouville-type lemma}
	In this section, we establish a Liouville-type result for the weighted Laplacian.
	\begin{lemma}\label{D2}
		Let $N\ge3$, $0\le a<\frac{N-2}{2}$ and $1<p\le\frac{N}{N-2-a}$. Let
		$u\in L^p(\R^N;|x|^{-ap})$ be a non-negative radial weak supersolution of
		\begin{equation*}
			-\operatorname{div}(|x|^{-2a}\nabla u)\ge0\quad\text{in }\R^N .
		\end{equation*}
		Then $u\equiv0$.
	\end{lemma}
	\begin{proof}
		Write $u=u(r)$. Since $u$ is a weak supersolution, $\int_0^\infty u'\varphi'\,
			r^{N-1-2a}\,dr\ge0$ for every non-negative radial $\varphi\in C_0^\infty(\R^N\setminus\{0\})$,
			so the function $h(r):=r^{N-1-2a}u'(r)$ is nonincreasing on $(0,\infty)$. Hence $u'$ changes
			sign at most once, and $\ell:=\lim_{r\to\infty}u(r)$ exists.
		
		If $\ell>0$, then $u\ge\ell/2$ for $r$ large, and since $p\le\frac{N}{N-2-a}$ gives
			$ap\le\frac{aN}{N-2-a}<N$,
		\begin{align*}
			\int_{\R^N}\frac{|u|^p}{|x|^{ap}}\dx\ \ge\ C\int_R^\infty r^{N-1-ap}\,dr=+\infty,
		\end{align*}
		contradicting $u\in L^p(\R^N;|x|^{-ap})$. Hence $\ell=0$.
		
		Suppose $u\not\equiv0$. If $h\equiv0$ on some $[r_1,\infty)$, then $u\equiv\ell=0$ there;
			as $h$ is nonincreasing, $h\ge0$ on $(0,r_1]$, so $u$ is nondecreasing there and $u(r_1)=0$
			forces $u\equiv0$, a contradiction. So there is $r_1>0$ with $h(r_1)=:-c<0$, and then
			$u'(r)\le-c\,r^{2a+1-N}$ for all $r\ge r_1$. Integrating from $r$ to $R$, letting
			$R\to\infty$ and using $u(R)\to0$ and $2a+2-N<0$,
		\begin{align*}
			u(r)\ \ge\ \frac{c}{N-2-2a}\,r^{2a+2-N}\qquad\text{for }r\ge r_1 .
		\end{align*}
		Consequently
		\begin{align*}
			\int_{\R^N}\frac{|u|^p}{|x|^{ap}}\dx\ \ge\ C\int_{r_1}^\infty r^{p(2-N+a)+N-1}\,dr,
		\end{align*}
		and this integral diverges precisely when $p(2-N+a)+N-1\ge-1$, that is, when
			$p\le\frac{N}{N-2-a}$. This contradicts $u\in L^p(\R^N;|x|^{-ap})$, so $u\equiv0$.\qed
	\end{proof}	

	\bibliographystyle{siam}
	\bibliography{ref}
\end{document}